\documentclass[11pt]{amsart}

\usepackage{amsthm}
\usepackage{amsmath}
\usepackage{amssymb}
\usepackage[margin=1in]{geometry}
\usepackage{enumerate}
\usepackage{hyperref}
\usepackage{mathrsfs}
\usepackage{color}
\usepackage{bm}

\title{Statistical properties for compositions of standard maps with increasing coefficient}
\author{Alex Blumenthal}
\date{\today}

\address{University of Maryland, College Park}
\curraddr{}
\email{alexb123@math.umd.edu}
\thanks{This research was supported by NSF Grant DMS-1604805.}



\theoremstyle{theorem}
\newtheorem{thm}{Theorem}[section]
\newtheorem{thmA}{Theorem}

\newtheorem{lem}[thm]{Lemma}
\newtheorem{prop}[thm]{Proposition}

\theoremstyle{definition}

\newtheorem{defn}[thm]{Definition}
\newtheorem{rmk}[thm]{Remark}

\newtheorem{cla}[thm]{Claim}
\newtheorem{exam}[thm]{Example}

\newcommand{\E}{\mathbb{E}}

\newcommand{\N}{\mathbb{N}}
\renewcommand{\P}{\mathbb{P}}
\newcommand{\R}{\mathbb{R}}

\newcommand{\Z}{\mathbb{Z}}

\newcommand{\Bc}{\mathcal{B}}
\newcommand{\Fc}{\mathcal{F}}
\newcommand{\Hc}{\mathcal{H}}

\newcommand{\Gc}{\mathcal{G}}

\newcommand{\Sc}{\mathcal{S}}
\renewcommand{\Hc}{\mathcal{H}}

\renewcommand{\a}{\alpha}
\renewcommand{\b}{\beta}

\newcommand{\e}{\epsilon}

\newcommand{\pd}{\partial}

\newcommand{\graph}{\operatorname{graph}}

\newcommand{\Id}{\operatorname{Id}}

\newcommand{\T}{\mathbb T}
\newcommand{\Cc}{\mathcal C}
\newcommand{\Leb}{\operatorname{Leb}}
\newcommand{\Lip}{\operatorname{Lip}}
\renewcommand{\graph}{\operatorname{graph}}

\newcommand{\Len}{\operatorname{Len}}
\newcommand{\Pc}{\mathcal P}

\newcommand{\Bor}{\operatorname{Bor}}

\newcommand{\modone}{\, (\text{mod } 1)}

\begin{document}

\maketitle

\begin{abstract}
The Chirikov standard map family is a one-parameter family of volume-preserving maps exhibiting hyperbolicity on a `large' but noninvariant subset of phase space. Based on this predominant hyperbolicity and numerical experiments, it is anticipated that the standard map has positive metric entropy for many parameter values. However, rigorous analysis is notoriously difficult, and it remains an open question whether the standard map has positive metric entropy for any parameter value. Here we study a problem of intermediate difficulty: compositions of standard maps with increasing parameter. When the coefficients increase to infinity at a sufficiently fast polynomial rate, we obtain a Strong Law, Central Limit Theorem, and quantitative mixing estimate for Holder observables. The methods used are not specific to the standard map and apply to a class of compositions of `prototypical' 2D maps with hyperbolicity on `most' of phase space.
\end{abstract}

\section{Introduction and statement of results}

Let $f : M \to M$ be a smooth dynamical system. In many systems of interest, the dynamics of $f$ does not tend to a stable or periodic equilibrium, as evidenced, e.g., when observables $\phi : M \to \R$ of such systems fluctuate indefinitely, i.e., $\phi \circ f^n(x)$ fluctuates as $n \to \infty$ for a `large' set of $x \in M$. In such cases, the asymptotic dynamics of the system is best described not by equilibria, but by a `physical' measure $\mu$ for $f$: an $f$-invariant probability measure $\mu$ on $M$ is called \emph{physical} if for a positive Lebesgue measure set of $x \in M$ (the `basin' of $\mu$) and any observable $\phi : M \to \R$, we have that 
\begin{align}\label{eq:SRB}
\lim_{n \to \infty} \frac1n \sum_{i = 0}^{n-1} \phi \circ f^i (x) = \int \phi \, d \mu \, .
\end{align}

Treating the sequence of observations $\{\phi \circ f^i\}_{i \geq 0}$ as a sequence of random variables, \eqref{eq:SRB} above is a Strong Law of Large Numbers. Pursuing this interpretation, it is natural to ask whether finer statistical properties hold, e.g.: 
\begin{itemize}
\item Central Limit Theorems pertaining to the convergence in distribution of $\frac{1}{\sqrt n} \sum_{i = 0}^{n-1} (\phi \circ f^i(X) - m)$, where $X$ is distributed in $M$ with some given law $\nu$ and $m \in \R$ is a centering constant; and
\item Decay of Correlations, i.e., estimates on the decay of $|\int \phi \circ f^n \cdot \psi \, d \mu - \int \phi \, d \mu \int \psi \, d \mu|$ as $n \to \infty$ for some class of observables $\phi, \psi$ on $M$.
\end{itemize}

These properties are by now classical for maps $f$ with uniform hyperbolicity, e.g., expanding, Anosov or Axiom A maps (see, e.g., \cite{liverani1995decay}). Outside the `uniform' setting, an extremely important tool in the exploration of statistical properties of deterministic dynamical systems is nonuniformly hyperbolic theory, also known as Pesin theory \cite{barreira2007nonuniform, young1995ergodic}. Assuming some control on the (typically nonuniform) rate of hyperbolicity, techniques have been developed for use in conjunction with nonuniform hyperbolicity to probe finer statistical properties of deterministic dynamical systems (e.g., the technique of countable Markov extensions, also known as Young towers \cite{young1998statistical}).

\medskip

\noindent {\it Difficulties and challenges. }

Use of these tools requires establishing nonuniform hyperbolicity, which is notoriously difficult to verify even for maps which `appear' to be hyperbolic on most (but not all) of phase space. In the volume-preserving category, the difficulties involved are exemplified by the Chirikov standard map family $\{ F_L\}_{L > 0}$ of volume-preserving maps on the torus $\T^2$ \cite{chirikov1979universal}. For large $L$, the map $F_L$ exhibits strong hyperbolicity (i.e., $F_L$ admits a continuous, invariant family of cones with strong expansion) on a large but noninvariant subset of phase space.  A key difficulty is that typical orbits will enter a set where cone invariance is violated (e.g. the vicinity of an elliptic fixed point for $F_L$), and the previously expanding invariant cone is potentially `twisted' towards the strongly contracting direction, after which all the growth accumulated may be destroyed.

\medskip

\noindent {\it Results in this paper. } 

In the interest of studying a problem of intermediate difficulty between the classical uniformly hyperbolic settings and the presently intractable two-dimensional nonuniformly hyperbolic setting exemplified by the Standard Map, we propose to study compositions of standard maps with \emph{increasing} coefficient. 
Cone twisting does occurs on a positive-volume subset of phase space at each timestep, and so we contend with many of the same problems described above for systems away from the `uniform setting'. Indeed, our hypotheses do not preclude the existence of elliptic fixed points for our compositions. Important for our analysis, however, is the fact that increasing the coefficient at each timestep both increases the strength of expansion and decreases the size of phase space committing `cone twisting'-- a crucial feature of this model is that a generic trajectory reaches these `bad' regions at most finitely many times when the increasing coefficients $\{ L_n\}$ are inverse summable (see \S\ref{subsec:basicConstruct}).

Our main results pertain to the situation when the sequence of coefficients increases sufficiently rapidly: we are able to establish a strong law of large numbers, a central limit theorem and decay of correlations (Theorems \ref{thm:strongLaw}, \ref{thm:CLT}, and \ref{thm:decayCorrelations} respectively). Our methods are quite flexible, and only rely on the bulk geometry of hyperbolicity on successively larger-volume subsets of phase space. As such, our results apply to a class of volume-preserving maps which are qualitatively similar to the standard map family. For this reason, the techniques of this paper are able to handle effectively `nonautonomous' dynamics, i.e., dynamics whose behavior is allowed to change with time.

Along the way towards proving the main results, certain `finite-time' decay of correlations estimates are obtained for standard maps with \emph{fixed} coefficient $L$, i.e., correlations estimates providing sharp bounds at all times $n \leq N_L$ (in our results, $N_L$ grows as a fractional power of $L$). This result (formulated as Theorem \ref{thm:finiteDoC}) is of independent interest: although it fails to be true \emph{asymptotic} result, these estimates demonstrate that for large $L$, the Standard map $F_L$ is strongly mixing on a relatively long timescale.

\medskip

\noindent {\it Related prior work. } 

The study of nonautonomous dynamical systems is still in its infancy, and many open questions remain. That being said, the statistical properties explored in this paper are closest to those on memory loss for nonautonomous compositions of hyperbolic maps \cite{bakhtin1995random1, bakhtin1995random2, stenlund2011non} (see also \cite{arnoux2005anosov}); Sinai billiards systems with slowly moving scatterers \cite{chernov2009brownian, stenlund2014vector, stenlund2012dispersing}; and  polynomial loss of memory for intermittent-type maps of the interval with a neutral fixed-point at the origin \cite{aimino2014polynomial, nicol2015central}. We have benefited especially from the techniques in \cite{conze2007limit}, which studies statistical properties of sequential piecewise expanding compositions in one dimension.

\smallskip

Pertaining to the Chirikov standard map, there is a large literature on this and related systems (e.g., Schroedinger cocycles) which we do not include here. See, e.g., the citations in \cite{blumenthal2017lyapunov} for a small sampling of such results.

\smallskip

Random dynamical systems can be thought of as a version of nonautonomous dynamics with some stationarity properties; see, e.g., \cite{arnold2013random, kifer2012ergodic}. Lyapunov exponents of random perturbations of the standard map with large coupling coefficient were studied in \cite{blumenthal2017lyapunov}. We also note \cite{dolgopyat2004sample}, which established quenched (samplewise) statistical properties for a large class of SDE in both the volume-preserving and dissipative regimens.

The analysis in this paper bears some qualitative similarities with that used in \cite{blumenthal2017dissipative}, which studies Lyapunov exponents and statistical properties of random perturbations of dissipative two-dimensional maps with qualitatively similar features to the Henon map; these results apply as well to the standard map. As it turns out, statistical properties of the corresponding Markov chain can be deduced from \emph{finite-time mixing estimates} for the dynamics, very much in keeping with the spirit of the analysis in the present paper (especially Theorem \ref{thm:finiteDoC}).

\smallskip

Lastly, we mention that the techniques in this paper may be useful in future studies of `bouncing ball' models of Fermi acceleration \cite{deSimoi2009, de2013fermi, dolgopyat2008bouncing}. As it turns out, the static wall approximation of bouncing ball models in a potential field gives rise to a Poincare return map bearing strong qualitative similarities to the standard map (see \cite{deSimoi2009} for a detailed derivation), and so it is conceivable that the analysis in this paper may shed insight on open problems related to ``escaping trajectories'' for such models.

\bigskip

\noindent {\it Acknowlegements.} 

\noindent The author thanks Dmitry Dolgopyat for suggesting this problem and for many helpful discussions.

\subsection{Statement of results}

\subsubsection*{Definition of model}\label{subsubsec:modelDefinition}

\noindent Let $M_0 \in \N, K_0, K_1 > 0$ be fixed constants. Let $L_0 > 0$, which should be thought of as sufficiently large, and let $\{ L_n \}$ be a nondecreasing sequence for which $L_0 \leq L_1 \leq L_2 \leq \cdots \leq L_n \leq \cdots$ for all $n$. In our results, we will assume that $L_n \to \infty$ at a sufficiently fast polynomial rate in $n$.

\medskip

For each $n \geq 1$, let $f_n : \T^1 \to \R$ be a $C^3$ function for which

\begin{itemize}
\item[(H1)] $\| f_n' \|_{C^1} = \| f_n' \|_{C^0} + \| f_n''\|_{C^0} \leq K_0 L_n$, 
\item[(H2)] $\Cc_n := \{ \hat x \in \T^1 : f_n'(\hat x) = 0\}$ is finite, with cardinality $\leq M_0$ , and
\item[(H3)] For any $n \geq 1, x \in \T^1$ we have $|f_n'(x)| \geq L_n K_1^{-1} d(x, \Cc_n)$.
\end{itemize}
%\noindent Condition (H2) is redundant and added only for emphasis. As one can check, (H1) and (H3) together imply that the critical points of $f_n$ are bounded away from each other by a constant $c_0 = c_0(K_0, K_1) > 0$ independent of $n$, hence $M_0 \leq \lceil c^{-1}_0 \rceil$.

\medskip

We will consider the nonautonomous composition of the maps $F_n : \T^2 \to \T^2$ defined by setting
\[
F_n = \bigg( \begin{array}{c} f_{n} (x) - y \modone \\ x \end{array} \bigg) \, .
\]
%where  $f_{n} : \T^1 \to \R$ is defined by setting $f_n = L_n \psi_n + a_n$. 
Above, $\modone$ refers to the projection $\R \to \T^1$ defined by $x \mapsto x - \lfloor x \rfloor$\footnote{Here for $x \in \R$ we define the floor function $\lfloor x \rfloor = \max\{ n \in \mathbb Z : n \leq x \}$.}, having abused notation somewhat and parametrized $\T^1$ by $[0,1)$. We will continue to use this convention throughout the paper.

\smallskip

We note that conditions (H1) -- (H3) are satisfied by the family $f_n(x) := L_n \sin(2 \pi x) + 2 x$, in which case $F_n$ is (up to conjugation by a linear toral automorphism) the Standard map with coefficient $L_n$. These conditions are also satisfied for the family $f_n(x) := L_n \psi(x) + a_n$ where $\{ a_n\} \subset [0,1)$ is any subsequence and $\psi : \T^1 \to \R$ is a map satisfying some $C^3$-generic conditions-- details are left to the reader. The hypotheses (H1) -- (H3) are similar to those for Theorem 1 in \cite{blumenthal2017lyapunov}.

\smallskip

For $n \geq m \geq 1$, we write $F^n_m = F_n \circ F_{n-1} \circ \cdots \circ F_m$, and write $F^n = F^n_1$. We adopt the conventions $F^{n-1}_n = \Id$, $F^0 = \Id$.

\subsubsection*{Results}

Our first result is a Strong Law of Large Numbers, which can be thought of as an ergodicity-type property for the nonautonomous compositions $\{ F^n\}$.

\begin{thmA}\label{thm:strongLaw}
Let $\a \in (0,1]$. Assume that $\{L_n\}$ is nondecreasing, and that $L_1 \geq L_0$, where $L_0  = L_0(K_0, K_1, M_0, \a) > 0$ is a constant. Let $\phi : \T^2 \to \R$ be $\a$-Holder continuous with $\int \phi \, d \Leb_{\T^2} = 0$.% and write $\hat S_N = \phi + \phi \circ F^1 + \cdots + \phi \circ F^N$ for $N \geq 0$.
\begin{itemize}
\item[(a)] If $N^2 L_{N}^{- \frac{\a}{3 \a + 4} }   \to 0$, then $\frac{1}{N} \sum_{i = 0}^N \phi \circ F^i \to 0$ in $L^2$.
\item[(b)] If $N^{4 + \e} L_{N}^{- \frac{\a}{3 \a + 4} }  \to 0$ for some arbitrary $\e > 0$, then $\frac{1}{N } \sum_{i = 0}^N \phi \circ F^i \to 0$ Lebesgue almost-everywhere.
\end{itemize}
\end{thmA}

\begin{exam}\label{exam1}
Fix $\a \in (0,1], p > 0$. Define $L_n = \max\{ L_0, n^p\}$ for some $p > 1$. Then, Theorem \ref{thm:strongLaw}(a) holds when $p > \a^{-1}(6 \a + 8)$ and (b) holds when $p \geq \a^{-1} (12 \a + 16)$. The results are optimal when $\a = 1$ (i.e. $\phi$ is Lipschitz); here $p \geq 14$ suffices for (a) and $p \geq 32$ for (b).
\end{exam}

Next is a central limit theorem for Holder observables.

\begin{thmA}\label{thm:CLT}
Let $\a \in (0,1]$. Let $\{ L_n \}$ be as in Theorem \ref{thm:strongLaw}, and additionally, assume
\[
\lim_{N \to \infty} N^{8} L_{N}^{- \frac{\a}{3 \a + 4}  } = 0 \, .
\]
Let $\phi$ be an $\a$-Holder continuous function on $\T^2$ for which $\int \phi \, d \Leb_{\T^2} = 0$. Let $X$ be a uniformly distributed $\T^2$-valued random variable. Then,
$\frac{1}{\sigma \sqrt{N} } \sum_{i = 0}^N \phi \circ F^i(X)$
converges in distribution to a standard Gaussian as $N \to \infty$ with
\[
\sigma^2 = \int \phi(x,y)^2 \, dx dy + 2 \int \phi(x,z) \phi(z,y) \, dx dy dz \, ,
\]
provided that $\sigma > 0$. Moreover, we have $\sigma = 0$ iff $\phi(x,y) = \psi(x) - \psi(y)$ for some continuous $\psi : \T^1 \to \R$.
\end{thmA}
\noindent These conditions are satisfied for $L_n$ as in Example \ref{exam1} when $p > 8  \a^{-1} (3 \a + 4)$. The asymptotic variance $\sigma$ appearing in Theorem \ref{thm:CLT} comes from an appropriate interpretation of the `singular' limit of the maps $F_n$ as $n \to \infty$. The condition $\phi(x,y) = \psi(x) - \psi(y)$ has the connotation of a coboundary condition for this singular limit. See the discussion in \S \ref{subsec:singularLimit} (in particular Remark \ref{rmk:singLimitVariance} and Lemma \ref{lem:coboundaryCondition}) for more details. Theorem \ref{thm:CLT} is proved in \S \ref{sec:CLT}. In the setting of Example \ref{exam1}, Theorem \ref{thm:CLT} holds when $p > \frac{8 (3 \a + 4)}{\a}$; the result is optimal when $\a = 1$, in which case $p > 56$ suffices.

\medskip

Finally, we present a decay of correlations estimate for the compositions $\{ F^n \}$.

\begin{thmA}\label{thm:decayCorrelations}
Fix $\eta \in (1/2, 1)$. Let $\{ L_n \}$ be a nondecreasing sequence for which $L_1 \geq L_0'$, where $L_0' = L_0'(K_0, K_1, M_0, \eta)$ is a constant, and assume that $\sum_n L_n^{-\frac12(1 -\eta)} < \infty$ for some fixed $\eta \in (1/2, 1)$. Then, there is a constant $C = C(K_0, K_1, M_0, \eta)$ for which the following holds.

Let $\a \in (0, 1]$ and let $\varphi, \psi$ be $\a$-Holder continuous functions on $\T^2$. Then,
\[
\bigg| \int \psi \circ F^n \cdot \varphi  - \int \varphi \int \psi  \bigg| \leq C \| \psi \|_\a \| \varphi \|_\a  \max \bigg\{ L_{\lfloor n/2 \rfloor}^{1 - 2 \eta} , \bigg( \sum_{i = \lfloor n/8 \rfloor}^\infty L_i^{- \frac12 (1 - \eta)} \bigg)^{\frac{\a}{\a + 2}} \bigg\}
\]
for all $n \geq 0$.
\end{thmA}
\noindent Above, all integrals $\int$ are with respect to $\Leb_{\T^2}$, and we have written
\[
[\varphi]_\a := \sup_{p, q \in \T^2} \frac{|\varphi(p) - \varphi(q)|}{d_{\T^2}(p,q)^\a} \, \quad \text{ and } \quad
\| \varphi \|_\a = \| \varphi \|_{C^0} + [\varphi]_\a \, 
\]
for $\a \in (0,1]$ are Holder moduli and norms, respectively, and $d_{\T^2}$ is the geodesic distance on $\T^2$ endowed with the flat geometry of $\R^2 / \Z^2$.

\begin{exam}\label{exam1}
Fix $\eta \in (1/2, 1)$. Define $L_n = \max\{ L_0',  n^p\}$ for some $p > 4$. In particular, $\sum_n L_n^{-\frac12(1 - \eta)} < \infty$ iff $p > 2/(1 - \eta)$. One obtains that the $\max \{ \cdots \} $ term on the right-hand side is
\[
\leq Const. \, \| \psi \|_\a \| \phi \|_\a n^{\max \{ p (1 - 2 \eta), \frac{\a}{\a + 2} (1 - \frac12 p (1 - \eta))\} }
\]
The exponent of $n$ is optimized at $\eta = \frac{3 \a p + 4 p - 2 \a}{5 \a p + 8 p}$ at the value $\frac{(4 - p) \a}{5 \a + 8}$ (valid since here $p > 2 / (1 - \eta)$ reduces to $p > 4$, which has been assumed), leading to the estimate
\[
\leq Const. \, \| \psi \|_\a \| \phi \|_\a n^{- \frac{\a(p - 4)}{5 \a + 8} } 
\]
The result is strongest when $\a = 1$, in which case decay of correlations is summable if $p > 17$.
%Assuming this and fixing $\a \in (0,1]$, the $\max \{ \cdots \}$ term on the right-hand side is
%\[
%\leq Const. \| \psi \|_\a \| \phi \|_\a L_0^{-  \frac{\a}{8(\a + 2)} } n^{- \frac{\a}{\a+2} (p/8 - 1)} \, .
%\]
%In particular, the decay of correlations term is summable if $p > 8 (  \frac{\a + 2}{\a} + 1) $.
\end{exam}

\bigskip

\noindent{\it Finite-time decay of correlations estimates for fixed-coefficient standard maps.}
\smallskip

\noindent Our estimates in this paper can also be used to obtain the following \emph{finite time} decay of correlations estimate for Holder observables.

\begin{thmA}\label{thm:finiteDoC}
Let $\a \in (0,1]$, and let $L \geq L_0''$, where $L_0'' = L_0''(\a) > 0$ is a constant. Let $\phi, \psi$ be $\a$-Holder-continuous functions on $\T^2$. Then,
\[
\bigg| \int \phi \circ F_L^n \cdot \psi - \int \phi \int \psi \bigg| \leq C \| \phi\|_\a \| \psi \|_\a \cdot n L^{- \frac{\a}{3 \a + 4}} \,. 
\]
for all $n \geq 2$, where $C = C(\a) > 0$ is a constant independent of $L,\psi, \phi$.
\end{thmA}
\noindent For each fixed $L > 0$, Theorem \ref{thm:finiteDoC} provides a nontrivial upper bound on correlations for times $n \ll L^{\frac{\a}{3 \a + 4}}$, and thus gives information on the mixing properties of the standard map in the so-called anti-integrable limit. Like before, the result is strongest at $\a = 1$.

%Our results here can be interpreted similarly: although the sequence of maps we study get more singular and do not converge in any reasonable sense as maps, the associated Koopman operators converge in a suitable weak sense to the transition operators of a certain Markov chain. See \S \ref{subsubsec:markovianLimit} for more details.

\subsubsection*{Plan for the paper}

We collection preliminaries and basic hyperoblicity results in \S \ref{sec:preliminaries}, with an emphasis on the geometry of iterates of curves roughly parallel to the strongly expanding direction (called \emph{horizontal curves}) for the dynamics. 

In \S \ref{sec:singLimit} we develop finite-time mixing estimates for the composition $\{ F^n\}$; this verifies Theorem \ref{thm:finiteDoC} and also lets us provide a statistical description of the `singular' limit of the maps $F_n$ as $n \to \infty$. In \S \ref{sec:SL} we deduce the Strong Law (Theorem \ref{thm:strongLaw}) and in \S \ref{sec:CLT} we prove the Central Limit Theorem (Theorem \ref{thm:CLT}). 

The proof of Theorem \ref{thm:decayCorrelations}, carried out in \S \ref{sec:shapeSet} and \S \ref{sec:decayCorrelations}, is logically independent of \S \ref{sec:singLimit} -- \S \ref{sec:CLT}; indeed, it should not be surprising that the `finite-time mixing' estimates in these sections do not yield the long-time asymptotic correlation estimate in Theorem \ref{thm:decayCorrelations}. The proof of the latter requires a more careful study of the `shape' of iterates of small, sufficiently nice sets $S \subset \T^2$. This is carried out in \S \ref{sec:shapeSet}, and the proof of Theorem \ref{thm:decayCorrelations} is completed in \S \ref{sec:decayCorrelations}.

% and is car
%These ideas are used in \S \ref{sec:shapeSet} to prove that for a small, sufficiently nice set $S \subset \T^2$, the image $F^n S$ is foliated `mostly' by horizontal curves of sufficiently long length. We apply this result to prove the decay of correlations in Theorem \ref{thm:decayCorrelations} in \S \ref{sec:decayCorrelations}. 
%
%In \S \ref{sec:SL} we obtain `finite-time' decay of correlations estimates-- estimates on the time $n$ correlation which hold on finite time horizons. Theorem \ref{thm:finiteDoC} follows on checking that these estimates hold when the coefficients are \emph{fixed}. These estimates are also applied to study the singular limit of the maps $F_n$ and to the proof of Theorem \ref{thm:strongLaw}. Theorem \ref{thm:CLT} is proved in \S \ref{sec:CLT} using a martingale approximation technique, applying the estimates in \S \ref{sec:SL}. 

%Logically, \S \ref{sec:shapeSet} \& \ref{sec:decayCorrelations} are independent from \S \ref{sec:SL} \& \ref{sec:CLT}

\subsubsection*{Notation and conventions}

\noindent We $\T^1$ parametrize as $[0,1)$ throughout the paper. The torus $\T^2$ carries the flat geometry of $\R^2 / \Z^2$, and we identify all tangent spaces with the same copy of $\R^2$. We write $d_{\T^1}, d_{\T^2}$ for the geodesic metrics on $\T^1, \T^2$ respectively.

We repeatedly use big-O notation: a quantity $\beta \in \R$ is said to be $O(\kappa)$ for some $\kappa > 0$, written $\beta = O(\kappa)$, if there is a constant $C > 0$, depending only on the \emph{system parameters} $K_0, K_1, M_0$, for which $|\a| \leq C \kappa$. Similarly, the letter $C$ is reserved for any positive constant depending only on the parameters $K_0, K_1, M_0$.

We write $\Leb$ or $\Leb_{\T^2}$ for the Lebesgue measure on $\T^2$, although unless otherwise stated, any integral $\int$ over $\T^2$ should be assumed to be with respect to Lebesgue. When $\gamma \subset \T^2$ is a $C^2$ curve, we write $\Leb_\gamma$ for the (unnormalized) induced Lebesgue measure on $\gamma$.

Lastly: the parameter $L_0 > 1$ is assumed fixed, and will be taken sufficiently large in a finite number of places in the proofs to come. Whenever $L_0$ is enlarged, it is done so in a way that depends only on the system parameters $K_0, K_1, M_0$ and the auxiliary parameter $\eta$ introduced below in \S \ref{subsubsec:basicConstruct}.

\bigskip

{\it From this point forward, we will assume that $\{L_n\}, \{f_n\}$ are as in (H1) -- (H3), and that $\{ L_n\}$ is a nondecreasing sequence.}

\section{Predominant hyperbolicity}\label{sec:preliminaries}

For all large $n$, the maps $F_n$ are \emph{predominantly hyperbolic}, which is to say that the derivative maps $d F_n$ exhibit strong expansion along roughly horizontal directions on an increasingly large (but non-invariant) proportion of phase space. Our purpose in this section is to make this idea precise and collect some preliminary results.

In \S \ref{subsec:basicConstruct} we essentially deal with hyperbolicity on the linear level: when $L_n \to \infty$ sufficiently fast, we show that the compositions $\{ F^n\}$ possess nonzero (in fact, infinite) Lyapunov exponents at Lebesgue-almost every point. On the other hand, the \emph{rate} at which this hyperbolicity is expressed is nonuniform across phase space, and so in analogy with standard nonuniformly hyperbolic theory in the stationary setting, we develop in \S \ref{subsec:basicConstruct} notion of \emph{uniformity set} to control this nonuniformity.

In \S \ref{subsec:horizontalCurves} and \ref{subsec:decayCorrelationsCurves}, we consider the nonlinear picture: the time evolution of curves roughly parallel to the unstable (horizontal) direction. The basic idea is that sufficiently long `horizontal curves' proliferate rapidly through phase space: this is precisely the mixing mechanism one anticipates when working with this model, and is used repeatedly throughout the paper. Standard hyperbolic theory preliminaries are given in \S \ref{subsec:horizontalCurves}, while in \S \ref{subsec:decayCorrelationsCurves} this mixing mechanism is more precisely laid out in the form of a mixing estimate for Lebesgue measure supported on a sufficiently long horizontal curve.

\subsection{Predominant hyperbolicity of maps $F_n$}\label{subsec:basicConstruct}

Let us begin by identifying subsets of phase space where the maps $F_n$ exhibit uniformly strong hyperbolicity. For $L > 0$ and $n \geq 1$, define the \emph{critical strips}
\[
S_{n, L} =  \{ (x,y) \in \T^2 : d(x, \Cc_n) \leq  K_1 L_n^{-1} L \} \, ,
\]
and note that by (H3), for $(x,y) \notin S_{n, L}$ we have $|f_n'(x)| \geq L$. For each $n$, outside $S_{n, L}$ we have that $F_n$ is strongly expanding in the horizontal direction: to wit, for any $L$ sufficiently large ($L \geq 10$ will do for our purposes) and any $n \geq 1, p \notin S_{n, L}$, the cone $C_h = \{ v = (v_x, v_y) \in \R^2 : |v_y| \leq \frac{1}{10} |v_x| \}$ is preserved by $(dF_n)_p$, and all vectors in the cone are expanded by a factor $\geq L/4$.

In particular, observe that $\Leb(S_{n, L}) \approx L / L_n$. Thus, for fixed $L$, the proportion of phase space $\T^2 \setminus S_{n, L}$ on which $F_n$ preserves and expands $C_h$ increases as $n$ increases. When the sequence $L_n$ increases sufficiently rapidly, this implies an infinite Lyapunov exponent almost everywhere:

\begin{lem}\label{lem:infiniteExponent}
Assume $\sum_{n = 1}^\infty L_n^{-1} < \infty$. Then, 
\begin{align}\label{eq:infiniteExponent}
\lim_{n \to \infty} \frac1n \log \| dF^n_p \| = \infty
\end{align}
for $\Leb$-almost every $p \in \T^2$.
\end{lem}
\begin{proof}
For each $L > 0$, we have
\[
\sum_{n = 1}^\infty \Leb (F^{n-1})^{-1} S_{n, L} = \sum_{n = 1}^\infty \Leb S_{n, L} \leq 2 K_1 M_0 L \sum_{n = 1}^\infty L_n^{-1} < \infty \, .
\]
The Borel Cantelli lemma thus applies to the sequence of sets $\{ (F^{n-1})^{-1} S_{n, L} \}_{n \geq 1}$, and so the set $S_L = \{ p \in \T^2 : F^{n-1} p \in S_{n, L} \, \, i. o. \}$ has zero Lebesgue measure. Taking $S = \cup_{N = 1}^\infty S_N$, it is now simple to check that \eqref{eq:infiniteExponent} holds for all $p \in \T^2 \setminus S$.
\end{proof}

Let us emphasize, however, that the limit \eqref{eq:infiniteExponent} is highly nonuniform in $x$, due to the fact that the critical strips $S_{n, L}$ have positive mass for all $n \geq 1$. We encode this nonuniformity in a way analogous to that of \emph{uniformity sets} (alternatively called \emph{Pesin sets}) for nonuniformly hyperbolic dynamics.

\subsubsection{Construction of uniformity sets for the composition $\{F^n\}$}\label{subsubsec:basicConstruct}

For our purposes in this paper, it is expedient to `fatten' the critical strips $S_{n, L}$ as follows. Let $\eta \in (0, 1)$, and for $n \geq 1$ define
\[
B_n(\eta) = \{ (x,y) \in \T^2 : d(x, \Cc_n) \leq  2 K_1 L_n^{-1 + \eta} \} \, .
\]
For $p = (x, y) \notin B_n(\eta)$, we have $|f_n'(x)| \geq 2 L_n^\eta$. In particular, for such $p$, we have that $(dF_n)_{p}$ preserves the cone $C_h$ and expands tangent vectors in $C_x$ by a factor $\geq L_n^\eta$. The parameter $\eta$ dictates the proportion of expansion we recover in $(B_n(\eta))^c$, hence the tradeoff: the larger $\eta$, hence the more expansion we demand away from the bad set $B_n(\eta)$, but the larger the bad sets $B_n(\eta)$ become.  We note that $\eta$ appears throughout the paper and is often fixed in advance; as such, for simplicity we often write $B_n = B_n(\eta)$.

% variation of the critical strips $S_{n, L}$ which we use throughout the paper. For $\eta \in (1/2,1)$ we define
%\[
%B_n = B_n(\eta) = \{ (x,y) \in \T^2 : d(x, \Cc_n) \leq  2 K_1 L_n^{-1 + \eta} \} \, ,
%\]
%and note that for $(x,y) \notin B_n$, we have $|f_n'(x)| \geq  2 L_n^\eta$ (the factor of $2$ is added to simplify later estimates). The parameter $\eta$ above should be regarded as fixed throughout (e.g., in the proof of Theorem \ref{thm:decayCorrelations} it is the same $\eta$ as in the statement of that result).

%\smallskip

%For each $n$, outside $B_n$ we have that $F_n$ is strongly expanding in the horizontal direction: the cone $C_x$ as before is preserved by $(dF_n)_p$ for all $p \notin B_n$, and expanded by a factor $\gtrsim L_n^{\eta}$. 

For $p \in \T^2$, define
\begin{align*}
\tau(p) &= 1 + \max \{ m \geq 1 : F^{m-1}p \in B_m \} \\
& = \min\{ k \geq 1 : F^{n-1}p \notin  B_n \,\, \text{ for all } n \geq k \} \, ;
\end{align*}
in particular, for a given orbit $\{ p_n = F^n p \}$, the derivative mapping $(dF_n)_{p_{n-1}}$ is uniformly expanding along the horizontal cone $C_h$ for all $n \geq \tau(p)$. In this way, the sets
\[
\Gamma_N = \{ \tau(p) \leq N\} \, .
\]
can be thought of as \emph{uniformity sets} for the composition $\{F^n\}_{n \geq 1}$. Repeating the proof of Lemma  \ref{lem:infiniteExponent} yields the following.

\begin{lem} \label{lem:borelCantelli}
Fix $\eta \in (0,1)$, and assume that $\sum_{n = 1}^\infty L_n^{-1 + \eta} < \infty$. Then, $\tau < \infty$ almost surely, and $\cup_N \Gamma_N$ has full Lebesgue measure.
%\[
%%\Leb \{(x,y) \in \T^2 : d(F^n(x,y), B_{n + 1}) \leq K_1 L_{n + 1}^{-1 + \eta} \text{ i. o.}\} = 0
%\Leb \{(x,y) \in \T^2 : F^n(x,y) \in  B_{n + 1})  \,\, \text{ for infinitely many } n \geq 0 \} = 0
%\]
%Here,
%\[
%\hat B_n := \{ (x,y) \in \T^2 : d(x, \Cc_n) \geq 2 K_1^{-1} L_n^{-1 + \eta} \} 
%\]
%are the same critical strips with twice the thickness.
\end{lem}
\noindent Indeed, we have the estimate
\[
\Leb \{ \tau > N \} \leq \sum_{n = N}^\infty \Leb(F^{n-1})^{-1} B_n = \sum_{n = N}^\infty \Leb B_n = O \bigg( \sum_{n = N}^\infty L_n^{-1 + \eta} \bigg) \, .
\]
%This explains, in part, the origin of the correlations estimate appearing in Theorem \ref{thm:decayCorrelations}.

%For $(x,y) \in \T^2$, define the time
%\begin{align*}
%\tau(x,y) &= 1 + \max \{ m \geq 1 : F^{m-1}(x,y) \in B_m \} \\
%& = \min\{ k \geq 1 : F^{n-1} (x,y) \in  B_n \,\, \text{ for all } n \geq k \} \, ,
%\end{align*}
%and the sets
%\[
%\Gamma_N = \{(x,y) \in \T^2 : \tau(x,y) \leq N\} \, .
%\]
%By Lemma \ref{lem:borelCantelli}, we have that $\tau < \infty$ almost surely, hence $\Leb(\cup_{N \geq 1} \Gamma_N) = 1$. Indeed, for the tail of $\tau$ we have the estimate
%\begin{align}\label{eq:tauTailEstimate}
%\Leb( \T^2 \setminus \Gamma_N) = \Leb \{ \tau > N\} \leq \sum_{n = N}^\infty \Leb_{\T^2}(B_n) = O\bigg(  \sum_{n = N}^\infty L_n^{-1 + \eta}\bigg) 
%\end{align}

\subsection{Horizontal curves}\label{subsec:horizontalCurves}

Curves roughly parallel to unstable directions, sometimes called $u$-curves in the literature, are an effective and well-used tool for describing the mixing mechanism of hyperbolic dynamical systems: the elongation of such curves under successive applications of hyperbolic dynamics leads to their proliferation through phase space, resulting in mixing. These ideas are standard for (autonomous) smooth dynamical systems exhibiting hyperbolicity; see, e.g., \cite{dolgopyat2005averaging, pesin1982gibbs, ruelle1976measure}.

In the setting of this paper, \emph{horizontal curves} play the role of $u$-curves. Although much of the material in this section is standard for iterates of a single map, we note that the maps $F_n$ in our compositions become more singular as $n$ increases. So, it is important to ensure that the necessary estimates (e.g. distortion control) do not worsen with $n$. For this reason, we re-prove below in \S \ref{subsec:horizontalCurves} what are otherwise standard results in hyperbolic dynamics.

\bigskip

The point of departure is an identification of a class of curves `roughly parallel to unstable (horizontal) directions'.

\begin{defn}\label{defn:horizCurve}
A \emph{horizontal curve} is a connected $C^2$ curve $\gamma \subset \T^2$ with the property that 
$\gamma = \{(x, h_\gamma(x)) : y \in I_\gamma \}$ for some (open, proper) subarc $I_\gamma \subset \T^1$
 and some Lipschitz continuous function $h_\gamma : I_\gamma \to \T^1$ with $\Lip h_\gamma \leq 1/10$.
% We say that a horizontal curve $\gamma$ is \emph{fully crossing} if $I_\gamma = \T^1$ (in which case, we allow $h_\gamma$ to be discontinuous at one point of $\T^1$).
\end{defn}

The plan is as follows. In Lemma \ref{lem:forwardsGT} below we describe the evolution of horizontal curves under successive iterates of our nonautonomous compositions $\{ F_m^n, m \leq n \}$ when these curves are assumed to avoid the critical strips $B_n$ for each $n$. Lemma \ref{lem:distControlHorizontalCurves} is a distortion estimate between trajectories evolving on the same horizontal curve. Finally, Lemma \ref{lem:multMixingPrep} considers the time evolution of \emph{sufficiently long} horizontal curves which are allowed to meet bad sets.

\bigskip

The following is description of the geometry of successive images of horizontal curves which do not meet the bad sets $\{ B_n\}$.

\begin{lem}[Forward graph transform]\label{lem:forwardsGT}
Fix $\eta \in (0,1)$; then, the following holds whenever $L_0$ is sufficiently large (depending on $\eta$). Let $N \geq 1$, and let $\gamma \subset \T^2$ be a $C^2$ horizontal curve of the form $\gamma = \gamma_N 
= \graph g_N = \{(x, g_N(x)) : x \in I_N\}$, where $I_N \subset \R$ and $g_N : I_N \to \R$ is
a $C^2$ function for which $\|g_N'\|_{C^0} \leq 1/10$ and $\|g_N''\|_{C^0} \leq 1$.

Let $n > N$, and assume 
that for all $N \leq k \leq n-1$, we have that \[F^{k-1}_N(\gamma) \cap  B_k = \emptyset \, .\] 

Then
for each $N \leq k \leq n$, we have that $\gamma_k =  F^{k-1}_N(\gamma)$ is a horizontal curve
of the form $\graph g_k = \{(x, g_k(x)) : x \in I_k\}$ for an interval $I_k \subset \T^1$ and a $C^2$ function
$g_k : I_k \to \T^1$ for which
\begin{itemize}
\item[(a)] We have the bounds
\[
\| g_k' \|_{C^0} \leq L_{k-1}^{- \eta} \quad \text{ and } \quad \| g_k'' \|_{C^0} \leq 2 K_0 L_{k - 1}^{- 3 \eta + 1} \, ; \text{ and}
\]
\item[(b)] for any $p_N^i \in \gamma, i = 1,2$, writing $F^{k-1}_N p_N^i = p_k^i$, we have that
\[
\| p_k^1 - p_k^2\| \leq L_k^{\eta} \| p_{k + 1}^1 - p_{k + 1}^2\| \, .
\]
\end{itemize}
\end{lem}
\begin{proof}
The proof is a standard graph transform argument, which we recall here. It suffices to describe the induction step, that is, the procedure for obtaining $g_{k+1}$ from $g_k$ for $N \leq k \leq n-1$. 

To start, define the `lifted' map $\tilde F_k: \T^2 \to \R \times \T^1$ by setting $\tilde F_k(x,y) = (f_k(x) - y, x)$ (that is, without the `$\modone$' in the first coordinate). Projecting $\tilde F_k(x, g_k(x))$ to the first coordinate results in a function $\tilde f_k : I_k \to \R$ of the form $\tilde f_k(x) = f_k(x) - g_k(x)$.

Since $\gamma_k \cap B_k = \emptyset$, we have $|f_k'| \geq 2 L_k^{\eta}$ on $I_k$, and so $|\tilde f_k'| \geq 2 L_k^\eta - 1/10 > 0$ (on taking $L_0 > 1$). In particular, $\tilde f_k : I_k \to \R$ is invertible on its image $\tilde I_{k+1}$. Defining $I_{k+1} \subset \T^1$ to be the projection of $\tilde I_{k+1}$ to $\T^1$, we define $g_{k+1} : I_{k+1} \to \T^1$ to be the (uniquely determined) mapping for which $g_{k+1} \big( \tilde f_k(x) \modone \big) = x$ for all $x \in I_k$. This completes the description of the induction step.

The estimates in item (a) is now derived from the implicit derivatives
\begin{gather*}
%g_k(x) = (\tilde{f}_{k-1} - g_{k - 1})^{-1}(x) \\
g_k'(x) = \frac{1}{(f_{k-1}' - g_{k - 1}')(g_k(x))} \quad \text{ and } \quad
g_k''(x) = - \frac{(f_{k-1}'' - g_{k - 1}'')}{(f_{k-1}' - g_{k - 1}')^3} (g_k(x))  \, .
\end{gather*}
The estimate in (b) follows from the bound $|(\tilde f_k)'| \geq 2 L_k^{\eta} - 1/10 \geq L_k^{\eta}$. All estimates require taking $L_0$ sufficiently large depending on $\eta$.
\end{proof}

Next we obtain distortion estimates along forward iterates of horizontal leaves in the setting of Lemma \ref{lem:forwardsGT}. %For simplicity, we introduce a new parameter $\eta' = (\eta + 1/2) / 2$, noting that $L^{\eta - \eta'}_0 \gg 1$ for $L_0$ sufficiently large.
\begin{lem}\label{lem:distControlHorizontalCurves}
Assume the setting of Lemma \ref{lem:forwardsGT}. Let $p_N^i \in \gamma, i = 1,2$, and write $p_n^i = F^{n-1}_N p_N^i$. Then
\[
\bigg| \log \frac{\| (dF^{n-1}_N)_{p_N^1} |_{T \gamma}\|}{\| (dF^{n-1}_N)_{p_N^2} |_{T \gamma}\|} \bigg| = O \big( L_N^{1 - 2 \eta} \|p_n^1 - p_n^2\| \big) \, .
\]
\end{lem}

\begin{rmk}
The above bound is quite poor unless $\eta \in (1/2, 1)$, which is why in Theorem \ref{thm:decayCorrelations}, and indeed throughout the paper, we will work exclusively in the setting where $\eta \in (1/2, 1)$. Of course, the lower the value of $\eta$, the stronger the decay of correlations estimate in Theorem \ref{thm:decayCorrelations}. It is likely that lowering $\eta$ is possible: one way to accommodate the distortion estimate in Lemma \ref{lem:distControlHorizontalCurves} is to further subdivide images of the curve $\gamma$ into pieces of size $\ll L_n^{1 - 2 \eta}$.
\end{rmk}

%{ \color{red} Better bound 
%
%\begin{align}
%\bigg| \log \frac{\| (dF^{n-1}_N)_{p_N^1} |_{T \gamma}\|}{\| (dF^{n-1}_N)_{p_N^2} |_{T \gamma}\|} \bigg|
%\leq C L^{-1 + \eta + \eta'}_N \| p_n^1 - p_n^2\| \, .
%\end{align}
%
%needed?
%}

\begin{proof}[Proof of Lemma \ref{lem:distControlHorizontalCurves}]
Write $p_k^i = F^{k-1}_N(p_N^i) = (x_k^i, y_k^i)$. Let $\gamma_k = F^{k-1}_N(\gamma)$ and $g_k : I_k \to \T^1, I_k \subset \T^1$ be as in Lemma \ref{lem:forwardsGT}. Then,
\[
\| (dF_k)_{p_k^i}|_{T \gamma_k} \| = \sqrt{\frac{1 + (g_{k + 1}'(x_{k + 1}^i))^2}{1 + (g_k'(x_k^i))^2}} |f_k'(x_k^i) - g_k'(x_k^i) | \, ,
\]
and so
\begin{align} \label{eq:distortionEstimateExpansion}
\log \frac{\| (dF_N^{n-1})_{p_N^1}|_{T \gamma} \|}{\| (dF_N^{n-1})_{p_N^2}|_{T \gamma} \|} = 
\frac12 \bigg(  \log \frac{1 + (g_N'(x_N^2))^2}{1 + (g_N'(x_N^1))^2} + \log \frac{1 + (g_n'(x_n^1))^2}{1 + (g_n'(x_n^2))^2}  \bigg) + \sum_{k = N}^{n-1} \log \frac{f_k'(x_k^1) - g_k'(x_k^1)}{f_k'(x_k^2) - g_k'(x_k^2)} \, .
\end{align}
For the first two terms, observe that for $\beta_1, \beta_2 \in [0, \infty)$, we have the elementary bound
$|\log (1 + \beta_1) - \log (1 + \beta_2)| \leq |\beta_1 - \beta_2|$,
and so for $k = N, n$, we have
\[
\bigg| \log \frac{1 + (g_k'(x_k^1))^2}{1 + (g_k'(x_k^2))^2} \bigg| \leq |(g_k'(x_k^1))^2 - (g_k'(x_k^2))^2| \leq 2 |g_k'(x_k^1) - g_k'(x_k^2)| \leq 2 \Lip(g_k'') \cdot |x_k^1 - x_k^2| \, .
\]
Applying the expansion estimate along images of horizontal curves as in Lemma \ref{lem:forwardsGT}(a),
\begin{align}\label{eq:expansionHorizCurves}
|x_k^1 - x_k^2| \leq L_k^{- \eta} |x^1_{k + 1} - x^2_{k + 1}| \leq \cdots \leq L_k^{- \eta} \cdots L_{n - 1}^{- \eta} |x^1_n - x^2_n| 
\end{align}
and the estimate $\Lip (g_k'') \leq 2 K_0 L_k^{1 - 3 \eta}$ coming from Lemma \ref{lem:forwardsGT}, we obtain the following upper bound for the first two terms in \eqref{eq:distortionEstimateExpansion}:
\[
\Lip(g_N'') \cdot |x_N^1 - x_N^2| + \Lip(g_n'') \cdot |x_n^1 - x_n^2| \leq 2 K_0 L_N^{1 - 3 \eta} (1 + L_N^{- (n-N) \eta}) |x_n^1 - x_n^2| \, . %\leq  |x_n^1 - x_n^2| \, 
\]
Thus these terms are $O(L_N^{1-3\eta})$.% \frac12 L_N^{-1 + 2 \eta'} \| p_n^1 - p_n^2 \|$ when $L_0$ is sufficiently large. %Here we are using $\Lip(g_n) \leq 1/10$ to compare $|x_n^1 - x_n^2|$ and $\| p_n^1 - p_n^2\|$.

We now estimate the summation term in \eqref{eq:distortionEstimateExpansion}. With $\tilde f_k = f_k - g_k : I_k \to \R$ as in the proof of Lemma \ref{lem:forwardsGT}, we have that
\[
|\log \tilde f_k'(x_k^1) - \log \tilde f_k'(x_k^2)| \leq   \frac{\sup_{\zeta \in I_k} |\tilde f_k''(\zeta)|}{\inf_{\zeta \in I_k} |\tilde f_n'(\zeta)|} \cdot |x_k^1 - x_k^2|  \leq 2 K_0 L_k^{1 - \eta} |x_k^1 - x_k^2|  \, . %\leq L_k^{1 - \eta'} |x_k^1 - x_k^2| 
\]
Applying \eqref{eq:expansionHorizCurves} and collecting,
\begin{align*}
\bigg| \log \frac{(\tilde f_N^{n-1})'(x^1_N)}{(\tilde f_N^{n-1})'(x^2_N)} \bigg| & \leq 
2 K_0 \bigg( \sum_{k = N}^{n-1} \frac{L_k^{1 - \eta}}{L_k^{\eta} L_{k + 1}^{\eta} \cdots L_{n-1}^{\eta}} \bigg) |x^1_n - x^2_n| \\
& \leq 2 K_0 L^{1 - 2\eta}_N \bigg( \sum_{k = N}^{n-1} L_N^{- (n - 1 - k) \eta}\bigg) |x_n^1 - x_n^2| \\ 
& \leq 3 K_0 L^{1 - 2\eta}_N \| p_n^1 - p_n^2\| %|x_n^1 - x_n^2| \\
%& \leq \frac12 L_N^{1 - 2 \eta'} \| p_n^1 - p_n^2\| \, .
%& \leq \frac13 |x_n^1 - x_n^2|
\end{align*}
when $L_0$ is taken suitably large. %Again we use $\Lip(g_n) \leq 1/10$ in the last line. 
This completes the estimate.
\end{proof}

The above results describe the dynamics of a horizontal curve $\gamma$ which `avoids' the bad sets $\{B_n\}$ for some amount of time. On the other hand, if a given horizontal curve is allowed to meet the bad sets along its trajectory, then we lose control over the geometry where these iterates meet bad sets. Below we describe an algorithm for excising those parts of a curve which fall into the bad set and describe the geometry of the parts of $\gamma$ with a `good' trajectory.

\medskip

We say that a horizontal curve $\gamma$ is \emph{fully crossing} if $I_\gamma = (0,1)$ (all notation here and below is as in Definition \ref{defn:horizCurve}).

\begin{lem}\label{lem:multMixingPrep}
Fix $\eta \in (1/2, 1)$. Let $\gamma$ be a horizontal curve. Then, for any $m \geq 1, k \geq m$, there is a set $\Bc_m^k(\gamma) \subseteq \gamma$ and a partition (possibly empty) of $\bar \Gamma_m^k(\gamma)$ of $F^k_m(\gamma \setminus \Bc_m^k(\gamma))$ into fully crossing curves with the following properties.
\begin{itemize}
\item[(a)] For any $\bar \gamma \in \bar \Gamma_m^k(\gamma)$, we have $\| h_{\bar \gamma}'\|_{C^0} \leq L_k^{- \eta}$.
\item[(b)] We have the estimate
\[
\Leb_\gamma( \Bc_m^k(\gamma)) = O\bigg(\sum_{i = m}^k L_i^{-1+ \eta} \bigg) \, .
\]
\item[(c)] For any $\bar \gamma \in \bar \Gamma_m^k(\gamma)$ and any $p, p' \in (F_m^k)^{-1}\bar \gamma$, we have
\[
\frac{\| (d F_m^k)_p|_{T \gamma}\| }{\| (d F_m^k)_{p'}|_{T \gamma}\| }
= 1 + O(L_m^{1 - 2 \eta}) 
\]
\end{itemize}
\end{lem}
\noindent When $k = m$, we write $\bar \Gamma_m(\gamma) = \bar \Gamma_m^m(\gamma), \Bc_m(\gamma) = \Bc_m^m(\gamma)$ for short.

\medskip

Observe that Lemma \ref{lem:multMixingPrep} is inherently limited in two ways: (i) it is a \emph{finite-time result}: for a given curve $\gamma$ and fixed $m \geq 1$, we have $\Bc_m^k(\gamma) = \gamma$ for all $k$ sufficiently large; and (ii) if $\gamma$ is too short, then we may even have $\gamma = \Bc_m(\gamma)$.

\begin{proof}[Proof of Lemma \ref{lem:multMixingPrep}]
Below, $\tilde F_m : \T^2 \to \R \times \T^1$ is as defined in the proof of Lemma \ref{lem:forwardsGT}. To start, we define $\bar \Gamma_m(\gamma), \Bc_m(\gamma)$ as follows. 

For each connected component $\gamma_i, 1 \leq i \leq k$, of $\gamma \setminus B_m$, the image $\tilde \gamma_i = \tilde F_m(\gamma_i)$ is of the form $\graph \tilde h_i$ where $\tilde h_i : \tilde I_i \to \T^1$ for an interval $\tilde I_i \subset \R$ of the form $(a_i - r_i, b_i + s_i)$, where $a_i, b_i \in \Z, r_i, s_i \in [0,1)$. 

If $a_i = b_i$, then we set $\bar \Gamma_m(\gamma) = \emptyset$ and $\Bc_m(\gamma) = \gamma$, checking that if this is indeed the case, then $\Leb_\gamma(\gamma) = O(L_m^{-1 + \eta})$ follows.

When, $a_i < b_i$, we define $\bar \Gamma_m(\gamma)$ to be the collection of curves of the form $\graph \tilde h_i(\cdot + l)$ (projected to $\T^2$) for $l = a_i, \cdots, b_i - 1$. We set 
\[\Bc_m(\gamma) = ( \gamma \cap B_m) \cup \bigcup_{i =1 }^k (\tilde F_m)^{-1} \graph (\tilde h_i|_{(a_i - r_i, a_i) \cup (b_i, b_i + s_i)}) \, .
\]
For each curve of the form $\hat \gamma = (\tilde F_m)^{-1}(\graph \tilde h_i|_{(a_i - r_i, a_i)})$, we have
\[
\Leb_\gamma( \hat \gamma) = O(L^{- \eta}_m)
\]
since $\gamma_i \cap B_m = \emptyset$, and similarly for curves of the form $\hat \gamma = (\tilde F_m)^{-1}(\graph \tilde h_i|_{(b_i, b_i + s_i)})$. Combining this with the bound $\Leb_\gamma(\gamma \cap B_m) = O(L_m^{-1 + \eta})$, we conclude
\[
\Leb_\gamma(\Bc_m(\gamma)) = O(L_m^{-1 + \eta}) \, .
\]
Lastly, Item (c) holds for $k = m$ by Lemma \ref{lem:distControlHorizontalCurves}.

\medskip

Let us now describe the induction procedure for obtaining $\bar \Gamma_m^{l+1}(\gamma), \Bc_m^{l+1}(\gamma)$ $l < k$, assuming that $\bar \Gamma_m^l(\gamma)$ and $\Bc_m^l(\gamma)$ have been defined and that item (c) holds for $k = l$. We define
\begin{gather*}
\bar \Gamma_m^{l+1}(\gamma) := \bigcup_{\bar \gamma \in \bar \Gamma_m^l(\gamma)} \bar \Gamma_{l+1}(\bar \gamma) \, , \text{ and} \\
\Bc_m^{l+1}( \gamma) = \Bc_m^l(\gamma) \cup (F_m^l)^{-1} \bigcup_{\bar \gamma \in \bar \Gamma_m^l(\gamma)} \Bc_{l+1}(\bar \gamma) \, .
\end{gather*}

Repeating the above steps until step $l = k$, we have that $\bar \Gamma^k_m(\gamma)$ is comprised of fully crossing horizontal curves $\bar \gamma$ for which $\| h_{\bar \gamma}' \|_{C^0} \leq L_k^{- \eta}$. Item (c) similarly follows by the distortion estimate in Lemma \ref{lem:distControlHorizontalCurves}. 

It remains to estimate the size of $\Bc_m^k(\gamma)$. We have for each $m \leq l < k$ that
\[
\Leb_\gamma(\Bc_m^{l+1}(\gamma)) = \Leb_\gamma(\Bc_m^l(\gamma)) + \Leb_\gamma(F_m^l)^{-1} \bigcup_{\bar \gamma \in \bar \Gamma_m^l(\gamma)} \Bc_{l+1}(\bar \gamma) \, .
\]
For each $\bar \gamma \in \bar \Gamma_m^l(\gamma)$, we have $\Leb_{\bar \gamma} \Bc_{l+1}(\bar \gamma) = O(L_{l+1}^{-1+ \eta})$, and so 
\begin{align*}
\Leb_\gamma (F_m^l)^{-1} \bigcup_{\bar \gamma \in \bar \Gamma_m^l(\gamma)} \Bc_{l+1}(\bar \gamma)
&= \sum_{\bar \gamma \in \bar \Gamma_m^l(\gamma)} \Leb_{\bar \gamma} (F_m^l)^{-1} (\Bc_{l+1}(\bar \gamma))\\
&= (1 + O(L_m^{1 - 2 \eta})) \sum_{\bar \gamma \in \bar \Gamma_m^l(\gamma)}  \Leb_\gamma ((F_m^l)^{-1} \bar \gamma ) \cdot \frac{\Leb_{\bar \gamma}(\Bc_{l+1}(\bar \gamma))}{\Leb_{\bar \gamma}(\bar \gamma)} \\
&= O(L_{l+1}^{- 1 + \eta}) 
\end{align*}
having applied the distortion estimate in item (c) with $k = l$. This completes the estimate.
\end{proof}

\subsection{Decay of correlations for curves}\label{subsec:decayCorrelationsCurves}

The proliferation of horizontal curves throughout phase space is a mixing mechanism for our system. The estimates below justify this in the following sense: the Lebesgue mass along a given fully crossing horizontal curve spreads around throughout phase space in such a way as to appoximate Lebesgue measure very closely for Holder-continuous observables.

\begin{prop}\label{prop:CLTdocCurves}
Let $\eta \in (1/2, 1)$. Assume $L_1 \geq \bar L_0$, where $\bar L_0 = \bar L_0(M_0, K_0, K_1, \eta)$. Let $\gamma$ be a fully crossing horizontal curve, and let $\psi : \T^2 \to \R$ be $\a$-Holder continuous. For $1 \leq m \leq n$, we have
\[
\bigg| \int_\gamma \psi \circ F_m^n d \Leb_\gamma - \Len(\gamma) \cdot \int \psi  \bigg| \leq C \| \psi\|_\a \bigg( L_{n}^{- \a (1-\eta)/(\a + 2)} + L_m^{1-2 \eta} + \sum_{k = m}^{n-1} L_k^{-1 + \eta} \bigg)
\]
\end{prop}
\noindent Note that Proposition \ref{prop:CLTdocCurves} does not stipulate any conditions on the summability of the tail of $\{ L_n \}$.

\begin{proof}

With $\psi$ fixed, $\gamma$ a fully crossing horizontal curve, let $K \in \N$, to be specified shortly, and let $\ell = K^{-1}$. 

Let $ I_1, \cdots,  I_K$ denote the partition of $[0,1)$ into $K$ intervals of length $\ell$ each. For $1 \leq i , j \leq K$, let $R_{i,j} = I_i \times I_j$. Note that with $\psi_{i,j} = \inf \{ \psi(p) : p \in R_{i,j}\}$, we have
\[
\| \psi - \sum_{1 \leq i, j \leq K} \psi_{i,j} \chi_{R_{i,j}} \|_{L^\infty} = O(\ell^\a \| \psi \|_\a) \, .
\] 
Thus
\[
(*) := \int_\gamma \psi \circ F_m^n d \Leb_\gamma = O(\ell^\a \| \psi\|_\a) + \sum_{1 \leq i,j \leq K} \psi_{i,j} \int_\gamma \chi_{R_{i,j}} \circ F_m^n \, d \Leb_\gamma \, .
\]
Form $\bar \Gamma_m^{n-1}(\gamma), \Bc_m^{n-1}(\gamma)$ as in Lemma \ref{lem:multMixingPrep}, so that for each $i,j$-summand, we have
\[
\int_\gamma \chi_{R_{i,j}} \circ F_m^n \, d \Leb_\gamma
=
O\bigg( \sum_{k = m}^{n-1} L_k^{-1 + \eta} \bigg) + \sum_{\bar \gamma \in \bar \Gamma_m^{n-1}(\gamma)} \int_{\bar \gamma} \chi_{R_{i,j}} \circ F_n \,  \frac{d \Leb_{\bar \gamma}}{\| dF^{n-1}_m\| \circ (F^{n-1}_m)^{-1} }  \, ,
\]
so that 
\[
(*) = \| \psi\|_\a \cdot O\bigg( \ell^\a + \sum_{k = m}^{n-1} L_k^{-1+ \eta} \bigg) + \sum_{1 \leq i,j \leq K} \psi_{i,j} 
\sum_{\bar \gamma \in \bar \Gamma_m^{n-1}(\gamma)} \int_{\bar \gamma} \frac{d \Leb_{\bar \gamma}}{\| dF^{n-1}_m\| \circ (F^{n-1}_m)^{-1} } \chi_{R_{i,j}} \circ F_n \, .
\]
By the distortion estimate in Lemma \ref{lem:multMixingPrep}(c), the $i,j,\bar\gamma$-summand equals
\[
(1 + O(L_m^{1-2\eta})) \cdot \Leb_\gamma((F_m^{n-1})^{-1} \bar \gamma) \underbrace{\int_{\bar \gamma} \chi_{R_{i,j}} \circ F_n \,  d \Leb_{\bar \gamma}}_{(**)} \, .
\]

To estimate $(**)$, observe that $\bar \gamma \cap F_n^{-1} R_{i,j} = \bar \gamma|_{j}$, where for a set $S \subset \T^2$ we write $S|_i = S \cap (I_i \times [0,1))$. Form now the collection $\bar \Gamma_n(\bar \gamma|_j)$ and the set $\Bc_n(\bar \gamma|_j)$. We obtain
\begin{align*}
(**) = \int_{\bar \gamma} \chi_{R_{i,j}} \circ F_n d \Leb_{\bar \gamma} &= O(\Bc_n(\bar \gamma|_{j})) + 
\sum_{\bar \gamma' \in \bar \Gamma_n(\bar \gamma|_j)} \int_{\bar \gamma'} \frac{d \Leb_{\bar \gamma'}}{\| d F_n|_{T \bar \gamma} \| \circ F_n^{-1}} \chi_{R_{i,j}} \\
&= O(L_n^{-1 + \eta}) +  (1 + O(L_n^{1 - 2 \eta})) \cdot  \sum_{\bar \gamma' \in \bar \Gamma_n(\bar \gamma|_j)} \Leb_{\bar \gamma'}(\bar \gamma'|_i) \cdot \Leb_{\bar \gamma}(F_n^{-1} \bar \gamma') \, 
\end{align*}
by the distortion estimate in Lemma \ref{lem:multMixingPrep}(c). Since $\| h_{\bar \gamma'}'\|_{C^0} \leq L_n^{- \eta}$ for each $\bar \gamma' \in \bar \Gamma_n(\bar \gamma|_j)$, we easily estimate $\Leb_{\bar \gamma'}(\bar \gamma'|_i) = (1 + O(L_n^{- \eta})) \ell$, so that
\begin{align*}
(**) &= O(L_n^{-1 + \eta}) + (1 + O(L_n^{1 - 2 \eta})) (1 + O(L_n^{- \eta})) \cdot \ell \cdot \Leb_{\bar \gamma}(\bar \gamma|_j \setminus \Bc_n(\bar \gamma|_j)) \\
& = O(L_n^{-1 + \eta}) + (1 + O(L_n^{1 - 2 \eta})) \cdot \ell \cdot \Leb_{\bar \gamma}(\bar \gamma|_j \setminus \Bc_n(\bar \gamma|_j)) \, .
\end{align*}

Now, $\Leb_{\bar \gamma} (\Bc_n(\bar \gamma|_j)) = O(L_n^{-1 + \eta})$, so 
\begin{align*}
\Leb_{\bar \gamma}(\bar \gamma|_j \setminus \Bc_n(\bar \gamma|_j)) &= \Leb_{\bar \gamma}(\bar \gamma|_j) + O(L_n^{-1 + \eta}) = (1 + O(L_{n-1}^{- \eta})) \ell + O(L_n^{-1 + \eta})  \\
& = \big(1 + O(L_{n-1}^{-\eta} + \ell^{-1} L_n^{-1 + \eta}) \big) \, \ell
\end{align*}
having used the estimate $\| h_{\bar \gamma}' \|_{C^0} \leq L_{n-1}^{- \eta}$. Consolidating our estimates,
\begin{align*}
(**) &= O(L_n^{-1 + \eta}) +  (1 + O(L_n^{1 - 2 \eta})) \cdot \big(1 + O(L_{n-1}^{-\eta} + \ell^{-1} L_n^{-1 + \eta}) \big) \cdot \ell^2 \\
& = (1 + O(L_n^{1 - 2 \eta} + L_{n-1}^{- \eta} + \ell^{-2} L_n^{-1 + \eta})) \ell^2 \, .
\end{align*}

This establishes the constraint $\ell^{-2} L_n^{-1 + \eta} \ll 1$. Plugging the above estimate back into the expression for $(*)$ and using this constraint gives
\begin{align*}
(*) &= \| \psi\|_\a \cdot O\bigg( \ell^\a + \sum_{k = m}^{n-1} L_k^{-1+ \eta} \bigg) \\
& + (1 + O(L_m^{1 - 2\eta} + L_{n-1}^{- \eta} + \ell^{-2} L_n^{-1 + \eta})) \Leb_\gamma( \gamma \setminus \Bc_m^{n-1}(\gamma)) \cdot \sum_{1 \leq i,j \leq K} \psi_{i,j} \ell^2 \\
&= \| \psi\|_\a \cdot O\bigg( \ell^\a + \sum_{k = m}^{n-1} L_k^{-1+ \eta} \bigg) \\
& + (1 + O(L_m^{1 - 2\eta}  + \ell^{-2} L_n^{-1 + \eta})) \big(\Len(\gamma) + O \bigg( \sum_{k = m}^{n-1} L_k^{-1 + \eta} \bigg) \big) \cdot \int \psi  \\
&= \| \psi\|_\a \cdot O\bigg( \ell^\a + \sum_{k = m}^{n-1} L_k^{-1+ \eta} \bigg) + (1 + O(L_m^{1 - 2\eta} + \ell^{-2} L_n^{-1 + \eta})) \Len(\gamma) \cdot \int \psi \\
& = \Len (\gamma) \cdot \int \psi + \| \psi \|_\a \cdot O \bigg(\ell^\a + \sum_{k = m}^{n-1} L_k^{-1+ \eta} + L_m^{1 - 2 \eta} + \ell^{-2} L_n^{-1 + \eta} \bigg) \, .
%& = \| \psi\|_\a \cdot O\bigg( \ell^\a + \sum_{k = m}^{n-1} L_k^{-1+ \eta} \bigg) \\
%& + (1 + O(L_n^{-1 + \eta} + L_{n-1}^{- \eta} + \ell^{-2} L_n^{-1 + \eta})) \int \psi \, d \Leb \\
\end{align*}

On setting $K = \ell^{-1} = \lceil L_{n}^{ (1 - \eta) / (\a + 2)} \rceil$, the proof is complete.
\end{proof}

%Here we begin to explore the ergodic and statistical properties of the compositions $\{ F^n \}$ by proving formulations of the Weak and Strong Law of Large Numbers in Theorem \ref{thm:strongLaw} when the sequence $\{ L_n \}$ increases sufficiently rapidly. 
%
%Our primary goal in this section is to prove certain finite-time decay of correlations estimates for sequences of maps in our framework. These estimates, formulated and proved in \S \ref{subsec:singularLimit}, are crucial to our approach to Theorems \ref{thm:strongLaw} and \ref{thm:CLT}; being of independent interest, they are re-formulated as Theorem \ref{thm:finiteDoC}. 
%
%In \S \ref{subsubsec:markovianLimit} we apply these estimates to express the singular limit of the $F_n$ as a certain Markov chain. As it is expedient to do so, we prove Theorem \ref{thm:strongLaw} in \S \ref{subsec:strongLaw}.  

\section{Singular limit of $\{ F_n \}$; finite time mixing estimates}\label{sec:singLimit}

Although the compositions $\{ F^n \}$ are \emph{nonautonomous} or `nonstationary' by design, we argue in this section that the individual maps $F_n$ do converge, in a sense to be made precise, to some \emph{stationary} process. This we formulate in a precise way in \S \ref{subsec:singularLimit}. As we argue below, these considerations naturally follow from finite-time mixing properties of the partial compositions $F_m^n$ for $m, n$ very large, $m \leq n$; we state and prove these mixing estimates in \S \ref{subsec:finiteTimeMixEst}, verifying the convergence mode described in \S \ref{subsec:singularLimit}.

As they are of independent interest, these finite-time mixing estimates are re-formulated for the standard maps $F_L, L > 0$ as Theorem \ref{thm:finiteDoC}. 

\subsection{Singular limit of $\{ F_n \}$}\label{subsec:singularLimit}

As $n$ increases, the maps $F_n (x,y)= (f_n (x) - y \modone, x)$ become more and more singular due to the fact that $L_n \to \infty$; in particular, $\lim_{n\to\infty} F_n$ does not exist in any meaningful topology on diffeomorphisms of $\T^2$. To motivate a meaningful convergence notion, let us consider the action in the $x$ coordinate given by the map $f_n : \T^1 \to \T^1$. 

Observe that for $n$ extremely large, $f_n: \T^1 \to \T^1$ is predominantly an expanding map, and so in one time iterate the value of $f_n (x), x \in \T^1$ is increasingly sensitive to $x \in \T^1$. Cast in a different light, $f_n$ is increasingly `randomizing' on $\T^1$, to the point where $x$ and $f_n(x)$ are increasingly decorrelated as $n \to \infty$. One might expect, then, that in the limit, $f_n (x)$ can be modeled by a random variable \emph{independent of} $x$. A step towards a precise formulation might be as follows: for some class of continuous observables $\phi, \psi : \T^1 \to \R$, we should expect that
\[
\lim_{n \to \infty} \int_{\T^1} \phi \circ f_n(x) \cdot \psi(x) = \int \phi \int \psi \, .
\]
Morally speaking, we expect that when $X$ is a random variable distributed in a `nice' way on $\T^1$, we have that the joint law of the pair $(X, f_n(X))$ converges, in some to-be-determined sense, to the joint law of a pair $(X,Z)$ for which $Z$ is independent of $X$.

\medskip

%A precise formulation is as follows: if $X$ is a $\T^1$-valued random variable with absolutely continuous distribution, then the random variable $f_n(X)$ converges \emph{in distribution} to an $X$-independent random variable $Z$ 
%
%
%To motivate a meaningful convergence notion for this situation, recall that \emph{any} diffeomorphism can be thought of as giving rise to a Markov chain on $\T^2$. In our case, $F_n : \T^2 \to \T^2$ is described as the Markov operator 
%\[
%Q_n(p, \cdot) := \delta_{F_n p} \, ,
%\]
%for $p \in \T^2$; here $\delta_{p}$ denotes the Dirac mass centered at $p \in \T^2$. 

Let us now return to the implications for the full maps $F_n : \T^2 \to \T^2$ and make things more precise. The above discussion motivates modeling $F_n$ for $n$ large by a \emph{Markov chain} $\{ Z_n = (X_n, Y_n) \}$ defined as follows.
Let $\b_1, \b_2, \cdots$ be IID random variables uniformly distributed on $\T^1$. Given an initial condition $Z_0 = (X_0, Y_0) \in \T^2$, we iteratively define
\[
Z_{n+1} = (X_{n+1}, Y_{n+1}) = ( \beta_{n+1}, X_n) \, . % \, , \quad \text{ and } \quad Y_{n+1} = X_n \, .
\]
for $n \geq 0$. The form of this Markov chain agrees with the idea, argued above, that $X, f_n(X)$ are ``asymptotically independent'' in the sense described above.

Let $P$ denote the transition operator associated with $Z_n$, so that 
\[
P((x,y), A \times B) = \Leb(A) \cdot \delta_x(B)
\]
for Borel $A, B \subset \T^2$, where $\delta_x$ denotes the Dirac mass at $x$. Write $P^k$ for the $k$-th iterate of $P$. For $\phi : \T^2 \to \R, k \geq 1$, we define $P^k \phi : \T^2 \to \R$ by $P^k \phi(x,y) = \int P^k((x,y), d \bar x d \bar y) \, \phi(\bar x, \bar y)$.

\begin{prop}\label{prop:singLimit}
Fix $k \geq 1$ and let $\phi, \psi : \T^2 \to \R$ be continuous. Assume $L_m \to \infty$ as $m \to \infty$. Then,
\[
\lim_{m \to \infty} \int \psi \circ F^{m+k-1}_m \cdot \phi = \int P^k \psi \cdot \phi \, .
\]
\end{prop}
%\begin{proof}
%Proposition \ref{prop:tailDoC} implies the desired limit when $\psi, \phi$ are Holder continuous. The full result follows from a density argument.
%\end{proof}
\noindent That is, the maps $F_n$ converge to the Markov chain $(Z_n)_n$ in the sense that the associated Koopman operators converge to the transition operator $P$ for Holder observables in a way reminiscent of the weak operator topology. Proposition \ref{prop:singLimit} is proved in \S \ref{subsec:finiteTimeMixEst} below.%The situation here is analogous to the nonautonomous CLT in \cite{conze2007limit} for compositions of $\beta$-transformations converging to a single $\beta$-transformation.

\begin{rmk}\label{rmk:singLimitVariance}
The convergence described in Proposition \ref{prop:singLimit} suggests that the asymptotic variance of sums $\frac{1}{\sqrt N} \sum_{i = 0}^{N-1} \phi \circ F^i$ as in the Central Limit Theorem (Theorem \ref{thm:CLT}) should coincide with the asymptotic variance $\hat \sigma^2(\phi)$ of $\frac{1}{\sqrt{N}} \sum_{i = 0}^{N-1} \phi(Z_i), Z_0 \sim \Leb_{\T^2}$. Developing the Green-Kubo formula for $\hat \sigma^2(\phi)$, we obtain
\begin{align*}
\hat \sigma^2(\phi) & = \E \big( \phi(Z_0)^2 \big) + 2 \sum_{l = 1}^\infty \E \big( \phi(Z_0) \, \phi(Z_l) \big) \\
& = \E \big( \phi(Z_0)^2 \big) + 2 \E \big( \phi(Z_0) \phi(Z_1) \big) \\
& = \int \phi^2 + 2 \int \phi(x,y) \phi(y,z) \, d x dy dz \, ,
\end{align*}
where we have used the fact that $Z_k, Z_0$ are independent when $k \geq 2$. This is precisely the form of $\sigma^2$ given in Theorem \ref{thm:CLT}. Here, $\E$ refers to the expectation where $Z_0 \sim \Leb_{\T^2}$.
%
%
% the reason for the form of the asymptotic variance appearing in {Theorem \ref{thm:CLT}}: for $\phi : \T^2 \to \R$ with $\int \phi = 0$, coinc same as the asymptotic variance $\hat \sigma^2(\phi) = \lim_{N \to \infty} \frac1N \sum_{i = 0}^{N-1} \E_\pi \big( \phi (Z_i)^2\big)$, where $\E_\pi$ denotes the expectation when $Z_0$ is disributed like the stationary measure $\pi  =\Leb_{\T^2}$. Developing the Green-Kubo formula for $\hat \sigma^2(\phi)$, we obtain
%\begin{align*}
%\hat \sigma^2(\phi) & = \E_\pi \big( \phi(Z_0)^2 \big) + 2 \sum_{l = 1}^\infty \E_\pi \big( \phi(Z_0) \, \phi(Z_l) \big) \\
%& = \E_\pi \big( \phi(Z_0)^2 \big) + 2 \E_\pi \big( \phi(Z_0) \phi(Z_1) \big) \\
%& = \int \phi^2 + 2 \int \phi(x,y) \phi(y,z) \, d x dy dz \, ,
%\end{align*}
%where we have used the fact that $Z_k, Z_0$ are independent when $k \geq 2$.
\end{rmk}

This perspective also explains the `coboundary condition' $\phi(x,y) = \psi(x) - \psi(y)$ for some bounded $\psi : \T^1 \to \R$. If $\phi$ has this form, then the sums in the CLT for this Markov chain telescope: $\phi(Z_0) + \phi(Z_1) + \cdots + \phi(Z_{n-1}) = - \psi(Y_0) + \psi(X_n)$, and so the asymptotic variance is zero. Let us now check that this is also a necessary  condition for the asymptotic variance $\hat \sigma^2(\phi)$ to be zero.

\begin{lem}\label{lem:coboundaryCondition}
Let $\phi : \T^2 \to \R$ be a Holder continuous function with $\int \phi \, d x dy = 0$. Then, $\hat \sigma^2(\phi) = 0$ iff $\phi(x,y) = \psi(x) - \psi(y)$, where $\psi : \T^1 \to \R$ is some Holder continuous function.
\end{lem}
\begin{proof}
We have the identity
\[
\hat \sigma^2(\phi) = \int \bigg( \phi(x,y) + \int \phi(z, x) dz - \int \phi (w, y) dw \bigg)^2 \, dx dy \, ,
\]
the verification of which is an elementary (albeit tedious) computation left to the reader. Now, $\hat \sigma^2(\phi) = 0$ implies $\phi(x,y) = \psi(x) - \psi(y)$ pointwise (since $\phi$ is continuous), where $\psi(x) : = - \int \phi(z, x) dx$.
\end{proof}

\subsection{Finite-time mixing estimates}\label{subsec:finiteTimeMixEst}

The limiting notion described in Proposition \ref{prop:singLimit} is at its core the statement that finite compositions $F_m^n, m \leq n$ are  `mixing' in the limit $m, n \to \infty$. We will, in fact, prove something much stronger: a concrete estimate on the correlation of $(x,y)$ to $F_m^n(x,y)$ for $m, n$ large.

\begin{prop}\label{prop:tailDoC}
Fix $\eta \in (1/2, 1)$ and $\a \in (0,1]$. Let $L_0$ be sufficiently large, depending on $\a, \eta$. Let $m \geq 1$ and let $\phi_1, \phi_2 : \T^2 \to \R$ be $\a$-Holder continuous functions. Then, there exists a constant $C > 0$, depending only on $K_0, K_1, M_0$, such that the following hold.
\begin{itemize}
\item[(a)] We have
\[
\bigg| \int \phi_1 \circ F_m \cdot \phi_2 - \int \phi_1(x,z) \phi_2(z,y) dx dy dz \bigg| \leq  C \| \phi_1 \|_\a \| \phi_2 \|_\a L_m^{- \min\{ 2\eta - 1, \frac{\a(1 - \eta)}{2 + \a} \} } 
\]
\item[(b)] Let $n > m$. Then,
\[
\bigg| \int \phi_1 \circ F^n_m \cdot \phi_2  \, - \int \phi_1 \int \phi_2 \bigg| \leq { C \| \phi_1 \|_\a \| \phi_2\|_\a \bigg( L_m^{- \min\{\a (1-\eta) / (2 + \a), 2 \eta - 1\}} + \sum_{k = m+1}^{n-1} L_k^{-1+ \eta} \bigg)} \, .
\]
\end{itemize}
\end{prop}

\noindent Observe that Proposition \ref{prop:singLimit} follows easily from Proposition \ref{prop:tailDoC}. Moreover, as we leave to the reader to check, the proof of Proposition \ref{prop:tailDoC} requires only that the sequence $\{L_n\}$ be \emph{nondecreasing}, and so applies equally well in the case when $L_m = L_{m+1} =\cdots = L_n = L$ for some fixed $L > 0$. Thus Theorem \ref{thm:finiteDoC} follows.

Items (a) and (b) are proved separately in \S \ref{subsubsec:proofTailDOCa}, \S \ref{subsubsec:proofTailDOCb} below, respectively.

\subsubsection{Proof of Proposition \ref{prop:tailDoC}(a)}\label{subsubsec:proofTailDOCa}

Throughout \S \ref{subsubsec:proofTailDOCa} and \S \ref{subsubsec:proofTailDOCb}, we let $I_i, R_{i,j}$ be as in the proof of Proposition \ref{prop:CLTdocCurves}, where $\ell = K^{-1}$ and $K \in \N$ will be specified at the end (twice, once for part (a) and again for part (b)). 

With $\a \in (0,1]$ and $\phi_1, \phi_2$ fixed, for $l = 1,2$ we define $\phi^l_{i,j} = \inf_{R_{i,j}} \phi_l$, so that 
\[\| \phi_l - \sum_{i,j} \phi_{i,j}^l \chi_{R_{i,j}} \|_{L^{\infty}} = O(\| \phi \|_\a \ell^\a) \, . \] 

To begin, we estimate
\begin{gather*}
\int \phi_1 \circ F_m \cdot \phi_2 = O(\| \phi_1 \|_\a \| \phi_2 \|_\a \ell^\a) + \sum_{1 \leq i, j, i', j' \leq K} \phi_{i,j}^1 \phi_{i',j'}^2 \int \chi_{R_{i,j}} \circ F_m  \cdot \chi_{R_{i', j'}}  \\
= O(\| \phi_1 \|_\a \| \phi_2 \|_\a \ell^\a) + \sum_{1 \leq i_0, i_1, i_2 \leq K} \phi_{i_2 i_1}^1 \phi_{i_1 i_0}^2  \int  \chi_{R_{i_1 i_0}} \chi_{R_{i_2 i_1}} \circ F_m \, ,
\end{gather*}
where in passing from the first line to the second we have used that $F_m(R_{i,j}) \subset [0,1) \times I_i$. 

Fixing $i_0, i_1, i_2$, let $y_0 \in I_{i_0}$ and set $H = I_{i_1} \times \{ y_0 \}$. Applying Lemma \ref{lem:multMixingPrep},
\begin{gather*}
(*) = \int_{H} \chi_{R_{i_2 i_1}} \circ F_m d \Leb_{H} = O\big(\Leb_{H} (\Bc_m(H)) \big) + \sum_{\bar \gamma \in \bar \Gamma_m(H)} \int_{\bar \gamma} \frac{d \Leb_{\bar \gamma}}{\| d F_m|_{T H} \| \circ F_m^{-1} } \, \chi_{R_{i_2 i_1}} d \Leb_{\bar \gamma} \\
= O(L_m^{-1 + \eta}) + \sum_{\bar \gamma \in \bar \Gamma_m(H)} (1 + O(L_m^{1 - 2 \eta}) )\Leb_{H} (F_m^{-1} \bar \gamma) \cdot \int_{\bar \gamma} \chi_{I_{i_2} \times [0,1)} \, d \Leb_{\bar \gamma} \, ,
\end{gather*}
having used again that $F_m(I_i \times [0,1)) \subset [0,1) \times I_i$ to develop the integrand on the far right. Estimating $\Leb_{\bar \gamma}(\bar \gamma \cap I_{i_2} \times [0,1)) = (1 + O(L_m^{- \eta})) \ell$ (having used that $\| h_{\bar \gamma}' \|_{C^0} = O(L_m^{- \eta})$), we obtain
\begin{gather*}
(*) = O(L_m^{-1 + \eta}) + (1 + O(L_m^{1 - 2 \eta })) (1 + O(L_m^{- \eta})) \Leb_H(H \setminus \Bc_m(H)) \cdot \ell \\
 = O(L_m^{-1 + \eta}) + (1 + O(L_m^{1- 2  \eta })) (1 + O(L_m^{- \eta})) (\ell + O(L_m^{-1 + \eta})) \cdot \ell \\
 =  \ell^2 \big(1 + O (L_m^{1 - 2 \eta} + \ell^{-2} L_m^{-1 + \eta} \big) \, .
\end{gather*}
Integrating over $y_0 \in I_{i_0}$, we conclude 
\[
\Leb(R_{i_0 i_1} \cap F_m^{-1} R_{i_2 i_1}) = \ell^3 (1 + O(\ell^{-2} L_m^{-1 + \eta} + L_m^{1-2\eta})) \, .
\]

Summing now over $1 \leq i_0, i_1, i_2 \leq K$ gives
\begin{gather*}
\int \phi_1 \circ F_m \cdot \phi_2  = O( \| \phi_1 \|_\a \| \phi_2 \|_\a (\ell^\a + \ell^{-2} L_m^{-1 +  \eta} + L_m^{1-2 \eta})) + \sum_{1 \leq i_0, i_1, i_2 \leq K}  \phi_{i_2 i_1}^1\phi_{i_1 i_0}^2 \ell^3 \\
= O( \| \phi_1 \|_\a \| \phi_2 \|_\a (\ell^\a + \ell^{-2} L_m^{- 1 + \eta} + L_m^{1-2 \eta})) + \int \phi_1(x,z) \phi_2(z,y) dx dy dz \, .
\end{gather*}
The proof is complete on setting $K = \ell^{-1} = \bigg\lceil L_m^{ \frac{1-\eta}{2 + \a}} \bigg\rceil$.

\subsubsection{Proof of Proposition \ref{prop:tailDoC}(b)}\label{subsubsec:proofTailDOCb}

All notation is as in the beginning of \S \ref{subsubsec:proofTailDOCa}. We estimate
\begin{gather*}
(**) = \int \phi_1 \circ F_m^n \cdot  \phi_2  = O( \| \phi_1\|_\a \|\phi_2\|_\a \ell^\a) + \sum_{1 \leq i,j \leq K} \phi_{i,j}^2 \int_{R_{i,j}} \,  \phi_1 \circ F_m^n \, .
\end{gather*}

Fix $1 \leq i,j \leq K$. For $y_0 \in I_j$, write $H = H(y_0) = I_i \times \{y_0\}$. Then
\[
\int_{R_{i,j}}  \phi_1 \circ F_m^n = \int_{y \in I_j} \int_{H(y_0)} \phi_1 \circ F^n_m d \Leb_{H(y_0)} \, d y_0 \, .
\]
Developing the inner integral and applying Lemma \ref{lem:multMixingPrep},
\begin{gather*}
\int_{H(y_0)} \phi_2 \circ F^n_m d \Leb_{H(y_0)} = O(\| \phi_1\|_0 \Leb_{H(y_0)}\Bc_m(H(y_0))) + \sum_{\bar \gamma \in \bar \Gamma_m(\gamma(y_0))} \int_{\bar \gamma} \frac{d \Leb_{\bar \gamma}}{\| dF_m|_{T H(y_0)}\| \circ F_m^{-1}} \phi_1 \circ F_{m+1}^n \\
= O(\| \phi_1\|_{C^0} L_m^{-1 +  \eta}) +(1 + O(L_m^{1 - 2 \eta}))  \sum_{\bar \gamma \in \bar \Gamma_m(H(y_0))} \Leb_{H(y_0)}(F_m^{-1} \bar \gamma) \int_{\bar \gamma} \phi_1 \circ F_{m+1}^n d \Leb_{\bar \gamma} \, .
\end{gather*}

The curves $\bar \gamma$ cross the full horizontal extent of $\T^2$ and so fall under the purview of Proposition \ref{prop:CLTdocCurves}. Applying the estimate there, we obtain
\begin{gather*}
\int_{\bar \gamma} \phi_1 \circ F_{m+1}^n d \Leb_{\bar \gamma} 
=
\Len(\bar \gamma) \, \int \phi_1 + O(\| \phi_1\|_\a \bigg( L_{n}^{- \a (1 - \eta)/(2 + \a)} + L_{m+1}^{1-2 \eta} + \sum_{k = m+1}^{n-1} L_k^{-1 + \eta} \bigg)) \\
= \int \phi_1 + O(\| \phi_1\|_\a \bigg( L_n^{- \a(1- \eta)/(2 + \a)} +  L_{m+1}^{1-2 \eta} + \sum_{k = m+1}^{n-1} L_k^{- 1+\eta} \bigg)) \, .
\end{gather*}
Summing over $\bar \gamma$ we obtain that $\int_{H(y_0)} \phi_1 \circ F_m^n d \Leb_{H(y_0)} $ equals
\begin{gather*}
O(\| \phi_1\|_{C^0} L_m^{-1+ \eta}) + (1 + O(L_m^{1-2\eta}))  \sum_{\bar \gamma \in \bar \Gamma_m(H(y_0))} \Leb_{H(y_0)}(F_m^{-1} \bar \gamma) \,  \cdot \\
\bigg( \int \phi_1 + O(\| \phi_1\|_\a \bigg( L_n^{- \a (1-\eta)/(2 + \a)} +  L_{m+1}^{1-2 \eta} + \sum_{k = m+1}^{n-1} L_k^{-1+ \eta} \bigg))\bigg) \\
= O(\| \phi_1 \|_\a \bigg( L_m^{-1+ \eta} + \ell (L_n^{- \a (1- \eta) / (2 + \a)} + L_{m+1}^{1 - 2 \eta} + \sum_{k = m+1}^{n-1} L_k^{- 1+\eta}) \bigg) ) +\\
 (1 + O(L_m^{1 - 2 \eta})) \Leb_{H(y_0)}(H(y_0) \setminus \Bc_m(H(y_0))) \int \phi_1 \\
= O(\| \phi_1 \|_\a \bigg( L_m^{- 1+\eta} + \ell (L_n^{- \a (1-\eta) / (2 + \a)} + L_{m+1}^{1 - 2 \eta} + \sum_{k = m+1}^{n-1} L_k^{-1+ \eta}) \bigg) ) + \\(1 + O(L_m^{1 - 2 \eta})) (1 + O(\ell^{-1}  L_m^{-1+ \eta})) \cdot \ell \int \phi_1 \\
%= \ell \cdot \bigg\{ O(\| \phi_2 \|_\a \bigg( \ell^{-1} L_m^{- \eta} +  L_n^{- \a \eta / 4} + L_{m+1}^{1 - 2 \eta} + \sum_{k = m+1}^{n-1} L_k^{- \eta} \bigg) ) + (1 + O(L_m^{1 - 2 \eta})) (1 + O(\ell^{-1}  L_m^{- \eta})) \cdot  \int \phi_2 \bigg\} \\
= \ell \cdot \bigg\{ O(\| \phi_1 \|_\a \bigg( \ell^{-1} L_m^{- 1+\eta} +  L_n^{- \a (1-\eta) / (2 + \a)} +  L_{m}^{1 - 2 \eta} + \sum_{k = m+1}^{n-1} L_k^{- 1+\eta} \bigg) ) +  \int \phi_1 \bigg\} \, .
\end{gather*}

Integrating over $y_0 \in I_j$ yields the same estimate for $\int \chi_{R_{i,j} } \phi_1 \circ F_m^n$ with an additional factor of $\ell$. Summing over $1 \leq i,j \leq K$, we have that $\int \phi_1 \circ F^n_m \cdot \phi_2$ equals
\begin{gather*}
 O(\| \phi_1 \|_\a \|\phi_2\|_\a \bigg( \ell^\a + \ell^{-1} L_m^{- 1+\eta} +  L_n^{- \a (1 - \eta) / (2 + \a)} + L_{m}^{1 - 2 \eta} + \sum_{k = m+1}^{n-1} L_k^{- 1+ \eta} \bigg) ) + \sum_{i,j = 1}^K \ell^2 \phi^2_{i,j} \int \phi_1 \\
= O(\| \phi_1 \|_\a \|\phi_2\|_\a \bigg( \ell^\a + \ell^{-1} L_m^{- 1+\eta} +  L_n^{- \a (1-\eta) / (2 + \a)} +  L_m^{1 - 2 \eta} + \sum_{k = m+1}^{n-1} L_k^{- 1+\eta} \bigg) ) + \int \phi_1 \int \phi_2  \, .
\end{gather*}

The proof is complete on setting $K = \lceil L_m^{(1-\eta)/(1 + \a)} \rceil$.

\section{Law of Large Numbers}\label{sec:SL}

We continue our study of the statistical properties of the composition $\{ F^n \}$ by proving Theorem \ref{thm:strongLaw}, a pair of formulations of a `law of large numbers' for time-averages of observables. %In the proof it costs us nothing to work with a sequence of observables instead of a single one. 

\medskip

{\it  In this section, $\a \in (0,1]$ is fixed, as are a sequence of $\a$-Holder continuous observables $\phi_i : \T^2 \to \R, i \geq 0$ with $\int \phi_i = 0$ for all $i$ and $\sup_{i \geq 0} \| \phi_i \|_\a \leq C_0$ for a constant $C_0 >0$.}

\bigskip

For $0 \leq M \leq N$, we define
\begin{gather*}
%\hat S_N = \phi_0 + \phi_1 \circ F_1 + \cdots + \phi_{N} \circ F^{N} \\
\hat S_{M, N} = \phi_M \circ F^M + \cdots + \phi_{N} \circ F^N 
\end{gather*}
and set $\hat S_N = \hat S_{0, N}$. Noting the simple estimate
\[
|\hat S_N - \hat S_{M, N}| = \bigg| \sum_{i = 0}^{M-1} \phi_i \circ F^i \bigg| \leq C_0 M 
\]
holds pointwise on $\T^2$, it follows that to prove a strong law for $\hat S_N$, it suffices to prove a strong law for $\hat S_{M, N}$ where $M = M(N) = \lfloor \sqrt N \rfloor$. Similarly, a weak law for $\hat S_N$ follows from a weak law for $\hat S_{M,N}$. More precisely, to prove Theorem \ref{thm:strongLaw} it suffices to prove the following.

\begin{prop}\label{prop:strongLaw}
For $N \geq 1$ let $M = M(N) = \lfloor \sqrt N \rfloor$. 
\begin{itemize}
\item[(a)] If $N^2 L_{N}^{- \frac{\a}{3 \a + 4} }  \to 0$, then $\frac{1}{N - M} \hat S_{M, N}$ converges in $L^2$ to $0$.
\item[(b)] If $N^{4 + \e} L_{\lfloor \sqrt N \rfloor}^{-  \frac{\a}{3 \a + 4} } \to 0$ as $N \to \infty$ for some $\e > 0$, then $\frac{1}{N - M} \hat S_{M, N}$ converges almost surely to $0$.
\end{itemize}
\end{prop}
%\noindent Theorem \ref{thm:strongLaw} follows on optimizing the $\min\{ \cdots \}$ exponent over $\eta \in (1/2, 1)$ at $\eta = \frac{2 \a + 2}{3 \a + 4}$. Hereafter we write 

\begin{proof}[Proof of Proposition \ref{prop:strongLaw}]

To start, we expand
\begin{gather*}
\int \hat S_{M, N}^2 = \sum_{n = M}^N \int \phi_n^2 \circ F^n_M + 2 \sum_{M \leq m < n \leq N} \int \phi_n \circ F^n_{m+1} \cdot \phi_m
\end{gather*}
For the first term, each summand is precisely $\int \phi_n^2 \leq C_0^2$. For the second term, the $m, n$ summand is bounded 
\[
\| \phi_n \|_\a \| \phi_m\|_\a \cdot O\bigg(L_{m+1}^{- \min \{ \a ( 1- \eta) / (2 + \a) , 2 \eta - 1\} } + \sum_{k = m+2}^{n-1} L_k^{-1 + \eta} \bigg) 
\]
by Proposition \ref{prop:tailDoC}(b), and so the entire summation is bounded
\begin{gather*}
C_0^2 (N-M)^2  O \bigg( L_M^{- \min \{ \a ( 1 - \eta) / (2 + \a) , 2 \eta - 1\} } + \sum_{k = M}^N L_k^{-1 + \eta} \bigg) \\
= C_0^2 (N-M)^3 O (L_M^{- \min \{ \a ( 1- \eta) / (2 + \a) , 2 \eta - 1 \} } ) 
\end{gather*}

Optimizing in $\eta$, the function $\eta \mapsto \min \{ \a (1 - \eta) / (2 + \a), 2 \eta - 1\}$ is maximized at the value $\frac{\a}{3 \a + 4}$ at the point $\eta = \frac{2 \a + 2}{3 \a + 4}$. Hereafter this value of $\eta$ is fixed.

\medskip

Setting $M = M(N) = \lfloor \sqrt N \rfloor$, we obtain that $N^{-2} \int \hat S_{M, N}^2 \to 0$ as $N \to \infty$ so long as $N L_{\lfloor \sqrt N\rfloor }^{- \frac{\a}{3 \a + 4}}  \to 0$, as we have in the hypotheses of item (a). For (b), our estimates imply that the sequence $\{ N^{-2} \int \hat S_{M, N}^2 \}_{N \geq 1}$ is summable whenever $N^{2 + \e} L_{\lfloor \sqrt N \rfloor}^{- \frac{\a}{3 \a + 4}}  \to 0$ for some $\e > 0$ (which we have from the condition in (b)). Summability implies fast convergence in probability, which implies almost sure convergence (using the Borel-Cantelli Lemma). This completes the proof.
\end{proof}

\section{Central limit theorem}\label{sec:CLT}

Here we carry out the proof of of the central limit theorem in Theorem \ref{thm:CLT}. A standard technique, attributed to Gordin, for proving the central limit theorem for a deterministic dynamical system is to look for \emph{reverse Martingale difference approximations} for sums of observables, and then to use probability theory tools for proving the Central Limit Theorem for sums of reverse Martingale differences (see, e.g., \cite{liverani1996central} for an exposition).

%We pursue here a slightly different method: we look instead for \emph{forward} Martingale difference approximations. The reasons for avoiding reverse Martingale differences are somewhat technical, but a major one is related to the fact that the necessary reverse filtration would be comprised of pieces of stable manifolds. Although local stable manifolds exist for the compositions $\{ F^n\}$, to the best of our knowledge they vary strictly continuously and not Holder continuously (due to the fact that the Lipschitz constants of the maps $F_n$ diverge to $\infty$; see, e.g., \cite{hirsch1977invariant} for more details on the Holder regularity of the stable foliation). 

We pursue a slightly different method: we construct here an array of \emph{forward} Martingale difference approximations. The corresponding \emph{forward} filtrations are comprised (mostly) of fully-crossing horizontal curves. The filtration is constructed in \S \ref{subsubsec:filtration}. Our martingale difference approximation is constructed in \S \ref{subsubsec:approxMartDiff}, and in \S \ref{subsubsec:deduceCLTmart} we show how the CLT for our approximation implies the CLT as in Theorem \ref{thm:CLT}. The CLT for our martingale difference approximation is proved in \S \ref{subsec:CLTmart}.

\medskip

\noindent {\it Throughout this section, $\a \in (0,1]$ is fixed, and $\phi : \T^2 \to \R$ is assumed to be an $\a$-Holder continuous observable with $\int \phi = 0$. The value $\eta \in (1/2, 1)$ is assumed fixed; as we did in the previous section, in \S \ref{subsubsec:deduceCLTmart} we will specialize to a particular value of $\eta$ depending on $\a$.}

\medskip

\noindent {\bf Notation: } We write $\E$ below for the expectation with respect to Lebesgue measure on $\T^2$. When $\Gc$ is a sub-sigma-algebra of the Borel sigma algebra, we write $\E(\cdot | \Gc)$ for the conditional expectation with respect to $\Gc$.

%In particular, this convergence reveals the reason for the form of the asymptotic variance appearing in {Theorem \ref{thm:CLT}}: for $\phi : \T^2 \to \R$ with $\int \phi = 0$, it is the same as the asymptotic variance $\hat \sigma^2(\phi) = \lim_{N \to \infty} \frac1N \sum_{i = 0}^{N-1} \E_\pi \big( \phi (Z_i)^2\big)$, where $\E_\pi$ denotes the expectation when $Z_0$ is disributed like the stationary measure $\pi  =\Leb_{\T^2}$. Developing the Green-Kubo formula for $\hat \sigma^2(\phi)$, we obtain
%\begin{align*}
%\hat \sigma^2(\phi) & = \E_\pi \big( \phi(Z_0)^2 \big) + 2 \sum_{l = 1}^\infty \E_\pi \big( \phi(Z_0) \, \phi(Z_l) \big) \\
%& = \E_\pi \big( \phi(Z_0)^2 \big) + 2 \E_\pi \big( \phi(Z_0) \phi(Z_1) \big) \\
%& = \int \phi^2 + 2 \int \phi(x,y) \phi(y,z) \, d x dy dz \, ,
%\end{align*}
%where we have used the fact that $Z_k, Z_0$ are independent when $k \geq 2$.

%Throughout, the parameter $\eta \in (1/2, 1)$ is fixed; when all the arguments are completed, Theorem \ref{thm:CLT} follows on optimizing $\eta \in (1/2, 1)$ at the value $\eta = \frac{2 \a + 2}{3 \a + 4}$.

\subsection{Preliminaries for CLT: Construction of a martingale approximation}\label{subsec:constructMart}

\subsubsection{Construction of the increasing filtrations $\{ \hat \Gc_{M, k}, k \geq M\}$}\label{subsubsec:filtration}

We will produce an increasing filtration of (most of) $\T^2$ by horizontal curves with a small and controlled exceptional set. Below, $M \in \N$ should be thought of as large.

First, we will construct a sequence of \emph{partitions} $\zeta_{M, M}, \zeta_{M, M+1}, \cdots, \zeta_{M, k}, \cdots$ of $\T^2$ with the following properties for each $M \leq k \leq N$:

\begin{itemize}
	\item[(A)] The partition $\zeta_{M, k}$ is ``mostly'' comprised of fully crossing horizontal curves; and
	\item[(B)] $\zeta_{M, k} \leq F_k^{-1} \zeta_{M, k + 1}$\footnote{Here ``$\leq$'' refers to the partial order on partitions: two partitions $\zeta, \zeta'$ satisfy $\zeta \leq \zeta'$ if any atom of $\zeta$ is a union of $\zeta'$ atoms.
}.
\end{itemize}
Once the $\zeta_{M, k}$ are constructed, we define $\Gc_{M, k}$ to be the sigma algebra of measurable unions of elements in $ \zeta_{M, k}$, and finally,
\[
\hat \Gc_{M, k} = (F_M^{k })^{-1}  \Gc_{M, k+ 1} \, ,
\]
so that $\{\hat \Gc_{M, k}\}_{k \geq M}$ is an increasing filtration on $\T^2$. This is the filtration we will use in the sequel to construct our forward Martingale difference approximation.

\subsubsection*{Construction of $\{\zeta_{M, k}, k \geq M\}$ satisfying (A), (B)}

Set $ \zeta_{M, M}$ to be the partition of $\T^2 \setminus \{ x = 0 \}$ into horizontal line segments. Applying Lemma \ref{lem:multMixingPrep}, for each $\zeta \in \zeta_{M, M}$ form $\Bc_M(\zeta)$ and $\bar \Gamma_M(\zeta)$, writing 
\[
G_{M, M+1} = \bigcup_{\substack{\zeta \in \zeta_{M, M} \\ \bar \zeta \in \bar \Gamma_M(\zeta) }} \bar \zeta \, , \quad B_{M, M+1} = \bigcup_{\substack{\zeta \in \zeta_{M, M} }} F_M(\Bc_M(\zeta)) \, .
\] 
Defining the partition $\Hc_{M, M +1} = \{G_{M, M+1}, B_{M, M+1}\}$, we now define the partition $\zeta_{M, M+1} \geq \Hc_{M, M+1}$ as follows: 
\begin{gather*}
\zeta_{M, M+1}|_{G_{M, M+1}} = \{ \bar \zeta : \bar \zeta \in  \bar \Gamma_M(\zeta), \zeta \in \zeta_{M, M}\} \, , \\
 \zeta_{M, M+1}|_{B_{M, M+1}} = \{ F_M(\zeta) \cap B_{M, M+1} : \zeta \in \zeta_{M, M} \} \, .
\end{gather*}
%Conceptually, one should regard $\zeta_{M, M+1}$ as a partition of the domain of $F_{M+1}$.

% $G_{M, M+1}$ it is the partition into curves from the collections $\bar \Gamma_M$ (and so consists of fully-crossing curves) and on $B_{M, M+1}$ it is the partition $F_M(\zeta_{M, M})$ restricted to $B_{M, M+1}$ (i.e. has atoms of the form $F_M(\zeta) \cap B_{M, M+1}$ for $\zeta \in \zeta_{M, M}$.

%We now define $\zeta_{M, M+1}$ on $\T^2$ to be the partition of $G_{M, M+1}$ into curves from the $\bar \Gamma_M$, while on 

Iterating, assume $\zeta_{M, k}$
% (a partition of, morally, the domain of $F_k$) 
 has been formed, where $k \geq M+2$, along with the partition $\mathcal H_{M, k} = \{G_{M, k}, B_{M,k}\}$ for which $\zeta_{M, k} \geq \mathcal H_{M, k}$. For each $\zeta \in \zeta_{M, k} |_{G_{M, k}}$ form $\bar \Gamma_k(\zeta)$ and define
\[
G_{M, k+1} = \bigcup_{\substack{\zeta \in \zeta_{M, k} \\ \bar \zeta \in \bar \Gamma_k(\zeta) }} \bar \zeta \, , \quad B_{M,k+1} = \bigcup_{\substack{\zeta \in \zeta_{M, k} }} F_k(\Bc_k(\zeta)) \, ,
\]
and define $\zeta_{M, k+1}$ 
%(a partition of the domain of $F_{k+1}$) 
by
\begin{gather*}
\zeta_{M, k+1}|_{G_{M, k+1}} = \{ \bar \zeta : \bar \zeta \in  \bar \Gamma_k(\zeta), \zeta \in \zeta_{M, k}|_{G_{M, k}}\} \, , \\
 \zeta_{M,k+1}|_{B_{M, k+1}} = \{ F_k(\zeta) \cap B_{M, k+1} : \zeta \in \zeta_{M, k} \} \, .
\end{gather*}

\noindent Below, we formulate and verify properties (A) and (B) above for the sequence $\zeta_{M, k}, k \geq M$ constructed above.
\begin{lem}\label{lem:sizeGoodSetCLT}
The partitions $\{ \zeta_{M, k}\}_{k \geq M}, \Hc_{M, k} = \{ G_{M, k}, B_{M, k}\}$ are measurable, and have the following properties for each $k \geq M$.
\begin{itemize}
\item[(a)] Every atom $\zeta \in \zeta_{M, k}|_{G_{M, k}}$ is a fully crossing horizontal curve for which $\|h_\zeta'\|_{C^0} \leq L_{k-1}^{- \eta}$.
\item[(b)] We have $\zeta_{M, k} \leq F_k^{-1} \zeta_{M, k+1}$.
\item[(c)] We have the estimate:
\[
\Leb(B_{M, k}) = O\bigg( \sum_{i = M}^{k-1} L_i^{-1 + \eta}\bigg) \, . 
\]
\end{itemize}
\end{lem}
\begin{proof}
Measurability is not hard to check. Items (a) and (b) follow from the construction. For the estimate in item (c), observe that for each $k \geq M$, $\zeta \in \zeta_{M, k}|_{G_{M, k}}$, we have $\Leb_\zeta(\Bc_k(\zeta)) = O( L_k^{-1+ \eta})$, hence $(\Leb_{\T^2})_\zeta(\Bc_k(\zeta)) \leq (1 + O(L_{k-1}^{- \eta})) \cdot O(L_k^{-1+ \eta}) = O(L_{k}^{- 1+\eta})$, where here $(\Leb_{\T^2})_\zeta$ is the disintegration measure of $\Leb_{\T^2}|_{G_{M, k}}$ with respect to $\zeta \in \zeta_{M, k}|_{G_{M, k}}$. We conclude
\[
\Leb(G_{M, k+1}) = (1 + O(L_{k}^{- 1+\eta})) \Leb(G_{M, k}) \, ,
\]
hence 
\[
\Leb(G_{M, m}) = \prod_{k = M}^{m-1} (1 + O(L_{k}^{-1+ \eta})) \geq 1 + O\bigg( \sum_{k = M}^{m-1} L_k^{-1+\eta}\bigg) \, . \qedhere
\]
\end{proof}

\bigskip

The choice of $\hat \Gc_{M, k}$ is made so that $F_M^{k-1} \hat \Gc_{M, k} = F_{k}^{-1} \Gc_{M, k+1}$ is a very `fine' sigma-algebra. Before proceeding, we record the following estimate. 
\begin{lem}\label{lem:finePartitionCLT}
Let $\phi$ be $\a$-Holder continuous, $k \geq M$. Then
\[
| \phi - \E(\phi | F_{k}^{-1} \Gc_{M, k+1}) | = O( \| \phi\|_\a L_k^{- \eta \a}) \, . 
\]
on $F_k^{-1} G_{M, k+1}$.
\end{lem}
\begin{proof}
Let $\zeta \in \Gc_{M, k+1}|_{G_{M, k+1}}$. Then $F_k^{-1} \zeta$ is, by our construction, a subsegment of a fully-crossing curve $\zeta' \in \zeta_{M, k} |_{G_{M, k}}$ with diameter $O(L_k^{- \eta})$. So, for any points $p, p' \in F_k^{-1} \zeta$, we have $|\phi(p) - \phi(p')| = O(\| \phi \|_\a L_k^{- \eta \a})$.
\end{proof}

\subsubsection{Approximation by sum of martingale differences}\label{subsubsec:approxMartDiff}

For a bounded observable $\phi : \T^2 \to \R$, convergence in distribution of $\frac{1}{\sqrt{N}} S_N(X), X \sim \Leb_{\T^2}$, where 
\[
S_N = \sum_{n = 1}^N \phi \circ F^{n-1} \, ,
\]
is equivalent to convergence in distribution of $\frac{1}{\sqrt{N}} S_{M, N}(X), X \sim \Leb_{\T^2}$, where 
\[
S_{M, N} = \sum_{n = M}^N \phi \circ F_M^{n-1} \, .
\]
and $M = M(N)$ is a sequence satisfying $M(N) \ll \sqrt{N}$. Here, ``$X \sim \Leb_{\T^2}$'' means that $X$ is a $\T^2$-valued random variable with law $\Leb_{\T^2}$. 

Thus, for Theorem \ref{thm:CLT}, it suffices to prove convergence in distribution of $\frac{1}{\sqrt N} S_{M, N}(X)$; for this, we approximate $S_{M, N}$ by a sum of Martingale differences with respect to the increasing filtrations $\hat \Gc_{M, k}, k \geq M$. 

\begin{prop}\label{prop:martApprox}
Let $M \leq N$. Define 
\[
\tilde S_{M, N} = \sum_{n = M }^N \E(\phi | (F_{n })^{-1} \Gc_{M, n + 1}) \circ F^{n-1}_M
= \sum_{n = M }^N \E(\phi \circ F^{n-1}_M | \hat \Gc_{M, n}) \, . 
\]
\begin{itemize}
\item[(a)] The sum $\tilde S_{M, N}$ admits the representation $\tilde S_{M, N} = \sum_{n = M}^N U_{M, N, n}$, where
\[
U_{M, N, n} = \sum_{m = n-1}^{N-1} \bigg(\E(\phi \circ F^m_M | \hat \Gc_{M, n}) - \E(\phi \circ F^m_M | \hat \Gc_{M, n-1})\bigg) \, .
\]
The sequence $\{ U_{M, N, n}, M \leq n \leq N\}$ is a forward Martingale difference adapted to $(\hat \Gc_{M, n}, M \leq n \leq N)$. Precisely, $\E(U_{M, N, n} | \hat \Gc_{M, n}) = U_{M, N, n}$ and $\E(U_{M, N, n} | \hat \Gc_{M, n-1}) = 0$.
\item[(b)] We have
\[
|S_{M, N} - \tilde S_{M, N}| = O\bigg((N-M) \| \phi \|_\a \sum_{m = M}^N L_m^{- \eta \a}\bigg)
\]
on $G_{M, N}$.
\end{itemize}
\end{prop}
\noindent Above, we use the convention that $\hat \Gc_{M, M-1} = \{ \emptyset, \T^2\}$ is the trivial sigma-algebra on $\T^2$. For notational simplicity, when $M, N$ are fixed we write $U_n = U_{M, N, n}$.

\begin{proof}
Item (b) is a simple consequence of Lemma \ref{lem:finePartitionCLT}. For item (a), the relation $\tilde S_{M, N} = \sum_{M \leq n \leq N} U_{M, N, n}$ can be verified by a direct computation.

Alternatively, following the analogue of the derivation of a reverse Martingale difference approximation given in \cite{conze2007limit} for forward martingale differences,  one can look for a Martingale difference $U_n = \E(\phi \circ F^n_M | \hat \Gc_{M, n} ) + h_n - h_{n + 1}$, where $(h_n)_{M \leq n \leq N + 1}$ is some sequence of ``coboundary'' functions to be determined. Making the ansatz $h_{N + 1} = 0$ and `solving' the conditions $\E(U_n | \hat \Gc_{M, n}) = U_n, \E(U_n | \hat \Gc_{M, n-1}) = 0$ for each $n$, we deduce formally that
\begin{align*}%\label{eq:defineHN}
h_n = - \sum_{m = n-1}^{N-1} \E \big(  \phi \circ F^m_M \big| \hat \Gc_{M, n-1} \big) \, .
\end{align*}
Plugging this formula into the relation $U_n = \E(\phi \circ F^n_M | \hat \Gc_{M, n} ) + h_n - h_{n + 1}$ yields the form of $U_n$ given above. The choice $\hat \Gc_{M, M-1} = \{ \emptyset, \T^2\}$ ensures that $h_M = 0$, hence $\tilde S_{M, N} = \sum_{M \leq n \leq N} U_n + h_M - h_{N + 1} = \sum_{M \leq n \leq N} U_n$ holds.
\end{proof}

%\newpage

% we approximate $S_{M, N}$ by a
%%\[
%%\tilde S_{M, N} = \sum_{n = M - 1}^N \E(\phi | (F_{n + 1})^{-1} \Gc_{M, n + 2}) \circ F^n_M
%%= \sum_{n = M - 1}^N \E(\phi \circ F^n_M | (F^{n + 1}_M)^{-1} \Gc_{M, n + 2}) 
%%= \sum_{n = M - 1}^N \E(\phi \circ F^n_M | \hat \Gc_{M, n}) \, . 
%%\]
%\[
%\tilde S_{M, N} = \sum_{n = M }^N \E(\phi | (F_{n })^{-1} \Gc_{M, n + 1}) \circ F^{n-1}_M
%%= \sum_{n = M }^N \E(\phi \circ F^{n-1}_M | (F^{n }_M)^{-1} \Gc_{M, n + 1}) 
%= \sum_{n = M }^N \E(\phi \circ F^{n-1}_M | \hat \Gc_{M, n}) \, . 
%\]
%The following is a simple consequence of Lemma \ref{lem:finePartitionCLT}.
%\begin{lem}\label{lem:finePartitionApprox}
%We have
%\[
%|S_{M, N} - \tilde S_{M, N}| = O\bigg((N-M) \| \phi \|_\a \sum_{m = M}^N L_m^{- \eta \a}\bigg)
%\]
%on $G_{M, N}$.
%\end{lem}

\subsubsection{Deducing {Theorem \ref{thm:CLT}} from the martingale approximation}\label{subsubsec:deduceCLTmart}

We will deduce {Theorem \ref{thm:CLT}} from the following.
\begin{prop}\label{prop:CLTmartDiff}
Assume $N^8 L_N^{- \frac{\a}{3 \a + 4}} \to 0$ as $N \to \infty$. For $N > 0$ let $M = M(N) = \lfloor \sqrt[4]{N} \rfloor$. Then
\[
\frac{1}{\sqrt{\sum_{n = M}^N \E U_{M, N, n}^2}} \sum_{n = M}^N U_{M, N, n}(X) \, , \quad X \sim \Leb_{\T^2}
\]
converges weakly to a standard Gaussian as $N \to \infty$.
\end{prop}
\noindent Proposition \ref{prop:CLTmartDiff} is proved in the next section. Let us first complete the proof of Theorem \ref{thm:CLT}. 

\bigskip

\noindent {\it Throughout, $M = \lfloor \sqrt[4]{N} \rfloor$. For the remainder of \S \ref{sec:CLT}, we specialize to the value $\eta = \frac{2 \a + 2}{3 \a + 4}$, noting that this value maximizes the function $\eta \mapsto \min \{ 2 \eta - 1, \a (1 - \eta)(\a + 2)\}$. In particular, $N^8 L_N^{- \min \{ 2 \eta - 1, \a ( 1- \eta)/(\a + 2)\}} \to 0$ as $N \to \infty$ under the conditions of Proposition \ref{prop:CLTmartDiff}.}

\bigskip

As we noted in at the beginning of \S \ref{subsubsec:approxMartDiff}, it suffices to prove the CLT for $\frac{1}{\sqrt N} S_{M, N}$, since here $M \approx N^{1/4} \ll \sqrt N$. Thus, to prove Theorem \ref{thm:CLT}, it suffices to check that

\begin{itemize}
\item[(I)] $\| S_{M, N} - \tilde S_{M, N}\|_{L^2} \to 0$ as $N \to \infty$, and
\item[(II)] $ \frac1N \sum_{n = M}^N \E U_{M, N, n}^2 \to \sigma^2$ as $N \to \infty$, where $\sigma^2$ is as in Theorem \ref{thm:CLT}.
\end{itemize}
\noindent For (I), we estimate $\| S_{M, N} - \tilde S_{M, N} \|_{L^2}$ as follows:
\begin{align*}
\| S_{M, N} - \tilde S_{M, N} \|_{L^2} & \leq C (N - M) \| \phi\|_0 \Leb(B_{M, N}) + C (N-M) \| \phi\|_\a \sum_{m = M}^N L_m^{- \eta \a} \\
& \leq C (N-M) \| \phi\|_\a \sum_{m = M}^N L_m^{- \min\{ \a\eta, 1 - \eta\} } %\leq C N^2 L_M^{- \eta \a}\, ,
\end{align*}
applying first Proposition \ref{prop:martApprox}(b) and then Lemma \ref{lem:sizeGoodSetCLT}. The above converges to $0$ as $N \to \infty$ by the {hypotheses of Proposition \ref{prop:CLTmartDiff}.} 

\smallskip

For (II), we observe
\[
\sum_{n = M}^N \E U_{M, N, n}^2 = \E \bigg( \sum_{n = M}^N U_{M, N, n} \bigg)^2 = \int_{\T^2} \tilde S_{M, N}^2 d \Leb_{\T^2} = \int_{\T^2} S_{M, N}^2 d \Leb_{\T^2} + O(\| S_{M, N} - \tilde S_{M, N} \|_{L^2}^2) \, .
\]
From (I), it follows that $\lim_{N \to \infty} \big( \|\tilde S_{M, N} \|_{L^2}^2 - \| S_{M, N}\|_{L^2}^2\big)  = 0$. It remains to compute $\| S_{M, N}\|_{L^2}^2$, which we do below.

\begin{lem}
Assume the setting of Proposition \ref{prop:CLTmartDiff}. With $M = M(N) = \lfloor \sqrt[4]{N}\rfloor$, we have
\[
\lim_{N \to \infty} \frac{1}{N} \int S_{M, N}^2 d \Leb = \sigma^2 =   \int \phi^2 + 2 \int \phi(x,z) \phi(z,y) dx dy dz \, ,
\]
\end{lem}
\begin{proof}
We have
\begin{align*}
\int S_{M, N}^2 &= (N - M + 1) \int \phi^2 + 2 \sum_{M \leq m < n \leq N} \int \phi \circ F_M^m \cdot \phi \circ F_M^n \\
& = (N-M+1) \int \phi^2 + 2 \sum_{n = M+1}^{N} \int \phi \cdot \phi \circ F_n + 2 \sum_{\substack{M \leq m < n \leq N \\ m < n - 1}} \int \phi \cdot \phi \circ F_{m+1}^n d \Leb \, .
\end{align*}

Applying Proposition \ref{prop:tailDoC}(a) to the middle summation, we obtain the estimate
\begin{gather*}
%2 (N - M) \int \phi(x,z) \phi(z,y) dx dy dz + O(\|\phi\|_\a^2 \sum_{k = M+1}^N L_k^{- \min\{ 2 \eta - 1, \frac{\a(1-\eta)}{2 + \a}\} }) \\
2 (N - M) \int \phi(x,z) \phi(z,y) dx dy dz + O(\|\phi\|_\a^2 (N-M) L_M^{- \min\{ 2 \eta - 1, \frac{\a(1-\eta)}{2 + \a}\} }) \, .
\end{gather*}
Applying Proposition \ref{prop:tailDoC}(b) to the $m, n$-summand in the third term,
\begin{gather*}
\int \phi \cdot \phi \circ F_{m+1}^n d \Leb = O(\| \phi\|^2_\a \bigg( L_{m+1}^{- \min\{ \a(1 - \eta)/ (2 + \a), 2 \eta - 1\}} + \sum_{k = m+2}^{n-1} L_k^{- 1+ \eta} \bigg) ) \\
= O(\| \phi\|_\a^2 (N-M) L_M^{- \min\{ \a (1-\eta) / (2 + \a), 2 \eta - 1\}})
\end{gather*}
and applying the summation, the third term is bounded
\[
O( \| \phi \|_\a^2 (N-M)^2 L_M^{- \min\{ \a (1-\eta) / (2 + \a), 2 \eta - 1\}} ) \, .
\]
All error terms go to $0$ under the hypothesis of Proposition \ref{prop:CLTmartDiff}.
\end{proof}

\subsection{Proof of Proposition \ref{prop:CLTmartDiff}}\label{subsec:CLTmart}

We use the following criterion for the CLT for arrays of martingale differences.
\begin{thm}[McLeish]\label{thm:mcl}
Let $(\Omega, \Fc, \P)$ be a probability space. Let $\{ k_n\}_{n \geq 1}$, be an increasing sequence of whole numbers tending to infinity, and for each $n \geq 1$, let $\Fc_{1, n} \subset \Fc_{2, n} \subset \cdots \subset \Fc_{k_n, n} \subset \Fc$ be an increasing sequence of sub-$\sigma$ algebras. For each such $n, i$, let $X_{i, n}$ be a random variable, measurable with respect to $\Fc_{i, n}$, for which $\E(X_{i, n} | \Fc_{i-1, n}) = 0$, and write $Z_n = \sum_{1 \leq i \leq k_n} X_{i, n}$. Assume
\begin{itemize}
\item[(a)] $\max_{i \leq k_n} |X_{i, n}|$ is uniformly bounded, in $n$, in the $L^2$ norm,
\item[(b)] $\max_{i \leq k_n} |X_{i, n}| \to 0$ in probability as $n \to \infty$, and
\item[(c)] $\sum_i X_{i, n}^2 \to 1$ in probability as $n \to \infty$. 
\end{itemize}
Then, $Z_n$ converges weakly to a standard Gaussian.
\end{thm}
We apply this to the array
\[
 \frac{1}{\sqrt{\sum_{m = M(N)}^N \E U_{M(N), N, m}^2}} U_{M(N), N, n}(X) \, , \quad M(N) \leq n \leq N\, , \quad X \sim \Leb_{\T^2} \, ,
\]
where as before $M(N) = \lfloor \sqrt[4]{N} \rfloor$.

A preliminary asymptotic estimate for $U_n$ is given in \S \ref{subsubsec:asymptoticUnEst}. The verification of (a) -- (c) as in Theorem \ref{thm:mcl} is given in \S \ref{subsubsec:verifyac}.

\subsubsection{An asymptotic estimate for $U_n$}\label{subsubsec:asymptoticUnEst}

The following approximation is extremely useful in the coming arguments.
\begin{lem}\label{lem:UnApprox}
Set $\hat U_n = U_n \circ (F_M^{n-1})^{-1}$. Then
\[
\hat U_n 
%\approx \phi - \E(\phi | \Gc_{M, n+1}) + \E(\phi | \Gc_{M, n+2}) \circ F_{n+1}
=  \phi - \psi + \psi \circ F_{n+1} + O(N^2 \| \phi \|_\a L_M^{- \min\{ \a (1-\eta) / (2 + \a), 2 \eta - 1\}}) %= \phi(x,y) - \psi(y) + \psi(x)
\]
with uniform constants on $F_{n}^{-1} G_{M, n+1}$, independently of $n$, where $\psi(y) = \int \phi(\bar x, y) d \bar x$.
\end{lem}
\begin{proof}
We have
\begin{align}\label{eq:approxUn}\begin{split}
\hat U_n  %\bigg( \sum_{m = n-1}^{N-1} \E(\phi \circ F^m_{n} | F^{n-1}_M \hat \Gc_{M, n}) - \E(\phi \circ F^m_{n} | F^{n-1}_M \hat \Gc_{M, n-1}) \bigg) \circ F^{n-1}_M \\
%& = \sum_{m = n-1}^{N-1} \E(\phi \circ F^m_{n} | F_n^{-1} \Gc_{M, n+1}) - \E(\phi \circ F^m_{n} |  \Gc_{M, n}) \\
&=\E(\phi | F_n^{-1} \Gc_{M, n+1}) - \E(\phi | \Gc_{M, n}) +  \E(\phi | \Gc_{M, n+1}) \circ F_n \\
& +  \sum_{m = n+1}^{N-1} \E(\phi \circ F^m_{n+1} | \Gc_{M, n+1}) \circ F_n + \sum_{m = n}^{N-1} \E(\phi \circ F_n^m | \Gc_{M, n})
\end{split}\end{align}
As we will show, the terms in the top line approximate to $\phi - \psi + \psi \circ F_{n+1}$, while the terms in the second line are small.

 For the first term in \eqref{eq:approxUn}, we have from Lemma \ref{lem:finePartitionCLT} that $|\E(\phi | F_n^{-1} \Gc_{M, n+1}) - \phi| = O(\| \phi\|_\a L_n^{- \a \eta})$ on $F_n^{-1} G_{M, n+1}$. 

For the second term in \eqref{eq:approxUn}, we have that 
\[
\E(\phi | \Gc_{M, n}) = \frac{1}{\Len(\gamma)} \int_\gamma \phi \, d \Leb_\gamma
\]
on $G_{M,n}$, where $\gamma$ is a fully crossing horizontal curve with $\| h_\gamma'\|_{C^0} \leq L_{n-1}^{- \eta}$. Let now $p \in \gamma$, $p = (x_0, y_0)$. Noting $|\phi(x, h_\gamma(x)) - \phi(x, y_0)| \leq \| \phi \|_\a |h_\gamma(x) - h_\gamma(x_0)|^\a \leq C \| \phi\|_\a L_{n-1}^{- \a \eta}$, we have
\[
\frac{1}{\Len(\gamma)} \int_\gamma \phi d \Leb_\gamma = (1 + O(\| \phi\|_\a L_{n-1}^{-\eta})) \int_0^1 \phi(x, h_\gamma(x)) dx = (1 + O(\| \phi\|_\a  L_{n-1}^{- \a \eta})) \int_0^1 \phi(x,y_0) dx \, ;
\]
we therefore conclude
\[
| \E(\phi | \Gc_{M, n}) - \psi | \leq C \| \phi\|_\a L_{n-1}^{- \a \eta}
\]
on $G_{M, n}$. Similarly, for the third term in \eqref{eq:approxUn}, we obtain the bound
\[
| \E(\phi | \Gc_{M, n+1}) \circ F_n - \psi \circ F_n| \leq C \| \phi\|_\a L_{n}^{- \a \eta}
\]
on $F_n^{-1} G_{M, n+1}$. 

For the fourth term in \eqref{eq:approxUn}, we estimate from Proposition \ref{prop:CLTdocCurves} that on $G_{M,n}$,
\begin{gather*}
\E(\phi \circ F_n^m | \Gc_{M, n}) = \frac{1}{\Len(\gamma)} \int_\gamma \phi \circ F_n^m d \Leb_\gamma = O(\| \phi \|_\a \cdot \bigg( L_m^{- \a (1-\eta) / (2 + \a)} + L_n^{1-2\eta} + \sum_{k = n}^{m-1} L_k^{-1+ \eta} \bigg) ) \\
= O(\| \phi\|_\a (N-M) L_M^{- \min\{ \a (1- \eta) / (2 + \a), 2 \eta - 1\} })
\end{gather*}
for some $\gamma \in \zeta_{M, n}$. Estimating similarly the fifth term in \eqref{eq:approxUn}, we deduce that on $F_n^{-1} G_{M, n+1}$ the contribution of the fourth and fifth terms combined is
\[
O(\| \phi\|_\a (N-M)^2 L_M^{- \min\{ \a (1-\eta) / (2 + \a), 2 \eta - 1\} } ) \, .
\]

%$F^{n-1}_M \hat \Gc_{M, n} = F_n^{-1} \Gc_{M, n+1}$
%
%$F^{n-1}_M \hat \Gc_{M, n-1} = \Gc_{M, n}$
\end{proof}

\subsubsection{Verifying properties (a) -- (c) in Theorem \ref{thm:mcl}}\label{subsubsec:verifyac}

\begin{proof}[Properties (a) \& (b)]
By Lemma \ref{lem:UnApprox}, we have that on $(F^{N-1}_M)^{-1} G_{M, N}$,
\[
|U_n| = O(\| \phi\|_{C^0} + \| \phi\|_\a N^2 L_M^{- \min\{ \a (1 - \eta) / (2 + \a), 2 \eta - 1\}}) = O(\| \phi\|_\a N^2 L_M^{- \min\{ \a (1 - \eta) / (2 + \a), 2 \eta - 1\}})  \, ,
\]
which is uniformly bounded in $n, N$. Property (b) is now immediate, since $\Leb(G_{M, N}) \to 1$ as $N \to \infty$.

For property (a), off $(F_M^{N-1})^{-1} G_{M, N}$ we have
\[
|U_n| \leq C N \| \phi\|_\a \, ,
\]
so,
\[
\| \max_{M \leq n \leq N} |U_{M, N, n}| \|_{L^2} \leq C \| \phi\|_\a \cdot  N \sqrt{\Leb(G_{M, N}^c) }
+ C \| \phi \|_\a \, .
\]
Property (a) follows from the estimate $\Leb(G_{M, N}^c) = O(\sum_{M}^{N-1} L_k^{- 1+ \eta}) = O((N-M) L_{M}^{-1+ \eta})$ in Lemma \ref{lem:sizeGoodSetCLT}.
\end{proof}

Below is a formulation of property (c).
\begin{prop}[Strong law for $\{ U_n^2\}$]
We have
\[
\lim_{N \to \infty} \frac{\sum_{n = M}^N U_{M, N, n}^2}{\E \sum_{n = M}^N U_{M, N, n}^2} = 1
\]
in probability.
\end{prop}
\begin{proof}
We prove the stronger property of convergence in $L^2$. To start, we evaluate
\begin{align*}
\int \bigg(\sum_M^N U_n^2 - \sum_{M}^N \E(U_n^2) \bigg)^2 d \Leb &= \sum_{M \leq m, n \leq N} \int (U_n^2 - \E(U_n^2))(U_m^2 - \E(U_m^2)) d \Leb \\
&= \sum_{n = M}^N \big( \E(U_n^4) - \E(U_n^2)^2 \big) \\
& + 2 \sum_{M \leq m < n \leq N} \int (\hat U_m^2 - \E(U_m^2))(\hat U_n^2 - \E(U_n^2)) \circ F_m^{n-1} d \Leb \, . 
\end{align*}
We start with bounding $\E(U_n^2), \E(U_n^4)$. For $N$ sufficiently large, we have on $(F_M^{N-1})^{-1} G_{M, N}$ that $|U_n| = O(\| \phi\|_\a)$ by Lemma \ref{lem:UnApprox}, while on 
the complement we have $|U_n| = O(N \| \phi \|_\a)$, and so applying the estimate on $\Leb(G_{M, N}^c)$ we obtain
\begin{gather*}
\E(U_n^2) = O( \| \phi\|_\a^2 ( N^3 L_{M}^{-1+  \eta} + 1)) \quad \text{ and } \quad
\E(U_n^4) = O(\| \phi \|_\a^4 (N^5 L_{M}^{- 1+ \eta} + 1)) \, .
\end{gather*}
Thus the first summation is bounded like
\[
O(\| \phi \|_\a^4 N (N^5 L_{M}^{-1+ \eta} + 1)) \, .
\]

For the second summation, let us write $\phi_*(x,y) := \phi(x,y) - \psi(y) + \psi(x)$ in the notation of  Lemma \ref{lem:UnApprox}. Since this quantity appears repeatedly, let us also use the shorthand $c = \frac{\a}{3 \a + 4}$, noting that under the hypotheses of Theorem \ref{thm:CLT} we have that $N^2 L_M^{- c} \to 0$ as $N \to \infty$. We estimate  
\[
\hat U_n^2 - \phi_*^2 = (\hat U_n + \phi_*)(\hat U_n - \phi_*) = O(\| \phi\|_\a^2 N^2 L_M^{-c} ( 1 + N^2 L_M^{- c}) )  = O(\| \phi\|_\a^2 N^2 L_M^{-c})
\]
 on $(F^{N-1}_n)^{-1} G_{M, N}$ and so
\[
|\E(\hat U_n^2) - \E(\phi_*^2)|  \leq C N^3 \|\phi\|_\a^2 L_{M}^{-1+ \eta} + C \| \phi\|_\a^2 N^2 L_M^{- c} = O( \| \phi \|_\a^2 (N^3 L_M^{-1 + \eta} + N^2 L_{M}^{-c})  \, ,
\]
hence
\begin{align*}
|\E(\hat U_n^2) \E(\hat U_m^2) - \E(\phi_*^2)^2 | & \leq \E(\hat U_n^2) |\E (\hat U_m^2) - \E(\phi_*^2)| + \E(\phi_*^2) |\E(\hat U_n^2) - \E(\phi_*^2)| \\
& = O(\| \phi \|_\a^4  (1 + N^3 L_M^{-1 + \eta}) (N^3 L_M^{-1 + \eta} + N^2 L_M^{-c} ) )
%N^8 L_{M}^{- \min\{ \a (1-\eta) / (2 + \a), 2 \eta - 1\}}) \, .
\end{align*}
On $(F^{N-1}_m)^{-1} G_{M, N}$, we have
\begin{gather*}
|\hat U_m^2 \cdot \hat U_n^2 \circ F_m^{n-1} - \phi_*^2 \cdot \phi_*^2 \circ F_m^{n-1}|
\leq \hat U_m^2 |\hat U_n^2 \circ F_m^{n-1} - \phi_*^2 \circ F_m^{n-1}| + \phi_*^2 \circ F_m^{n-1} \cdot |\hat U_m^2 - \phi^2_*| \\
= O(\| \phi \|_\a^4 (1 + N^2 L_M^{-c}) N^2 L_M^{-c} ) = O( \| \phi\|_\a^4 N^2 L_M^{-c}) \, .
%= (\phi_*^2 - \E(\phi_*^2))( \phi_*^2 - \E(\phi_*^2)) \circ F_m^{n-1} + O(\cdots)
\end{gather*}
Collecting,
\begin{gather*}
\int \hat U_m^2 \hat U_n^2 \circ F_m^{n-1} - \E(U_m^2) \E(U_n^2) - \bigg( 
\int \phi_*^2 \cdot \phi_*^2 \circ F_m^{n-1} - \E(\phi_*^2)^2 \bigg) \\
=  O(\| \phi\|_\a^4 (1 + N^3 L_M^{-1 + \eta}) (N^3 L_M^{-1 + \eta} + N^2 L_M^{-c}) ) 
%N^8 L_{M-1}^{- \min\{ \a (1-\eta) / (2 + \a), 2 \eta - 1\}} ) \, .
\end{gather*}
Applying now Proposition \ref{prop:tailDoC}(b), we obtain the estimate
\[
\bigg| \int \phi_*^2 \cdot \phi_*^2 \circ F_m^{n-1} - \bigg(\int \phi_*^2\bigg)^2 \bigg| = O(\| \phi\|_\a^4 (N L_M^{-1 + \eta} + L_M^{-c})) \, ,
\]
% and noting $\| \phi_*\|_\a = O(\| \phi\|_\a)$,
so we conclude
\[
\int \hat U_m^2 \hat U_n^2 \circ F_m^{n-1} - \E(U_m^2) \E(U_n^2) = O(\| \phi\|_\a^4 (1 + N^3 L_M^{-1 + \eta}) (N^3 L_M^{-1 + \eta} + N^2 L_M^{-c}))
%N^8 L_{M}^{- \min\{ \a (1-\eta) / (2 + \a), 2 \eta - 1\}}) \, .
\]
Summing over the $\approx N^2$ terms and noting that $\big( \sum_{M}^N \E(U_n^2)\big)^2 \approx \sigma^4 N^2$ for $N$ large, we obtain
\[
\frac{1}{\| \phi\|_\a^4} \bigg\| \frac{\sum_M^N U_n^2}{\sum_M^N \E U_n^2} - 1 \bigg\|_{L^2}^2 =
O\big(   (1 + N^3 L_M^{-1 + \eta}) (N^3 L_M^{-1 + \eta} + N^2 L_M^{-c}) + N^{-1} + N^4 L_M^{-1 + \eta} \big) \, .
%N^8 L_{M}^{- \min\{ \a (1-\eta) / (2 + \a), 2 \eta - 1\}}) \, . \qedhere
\]
The proof goes through if all terms on the RHS go to $0$ as $N \to \infty$. For this, it suffices that $N^2 L_M^{-c} \to 0$ as $N \to \infty$: to see this, observe that $N^4 L_M^{-1 + \eta} \leq N^4 L_M^{-2c}$ holds for any $\eta \in (1/2, 1), \a \in (0,1)$. The latter clearly goes to $0$ when $N^2 L_M^{-c} \to 0$.
\end{proof}

\section{Hyperbolicity and the shape of successive iterates of a set}\label{sec:shapeSet}

We close this paper with the proof of Theorem \ref{thm:decayCorrelations}, given in \S \ref{sec:shapeSet} and \S \ref{sec:decayCorrelations}.

%Mixing properties can be obtained by showing that successive images of any given measurable set proliferate through phase space in a roughly uniform way. Accomplishing this typically requires some control on the rough geometry or shape of successive images of a set. 

We argued in \S \ref{sec:preliminaries} that fully-crossing horizontal curves proliferate throughout phase space in a roughly uniform way, and that this proliferation is the mixing mechanism for the compositions $\{ F^n \}$. In this section, we flesh out this picture by showing the following: given a set $S \subset \T^2$ with a suitably nice boundary and $n$ large enough, the $n$-th image $F^n(S)$ is `mostly' foliated by disjoint fully-crossing horizontal curves. 

%The aim of \S \ref{sec:shapeSet} is to achieve this by showing that for large enough $n$, the $n$-th image $F^n(S)$ of a small open set $S \subset \T^2$ can be foliated by a collection of horizontal curves, `most' of which are `long enough' for our purposes. The proliferation of `long' horizontal curves throughout phase space is precisely the mixing mechanism in the present setting; indeed, as we show in \S \ref{subsec:decayCorrelationsCurves} (Proposition \ref{prop:CLTdocCurves}), the image of Lebesgue measure on a fully-crossing curve closely approximates Lebesgue measure on $\T^2$ for Holder observables. Thus, achieving a foliation by `long enough' horizontal curves is a crucial component of our analysis.

\medskip

The plan is as follows. In \S \ref{subsubsec:constructHatGamma} we construct for each $n$ a foliation of $S_n = F^{n-1} S$ by horizontal curves. It is shown in \S \ref{subsubsec:defineSigma} that for $n$ sufficiently large, a large proportion of the curves in the foliation of $S_n$ are `sufficiently long', in the sense that in one timestep such curves become fully crossing.
In \S \ref{subsubsec:disintegrationNuN} we show that on disintegrating Lebesgue measure restricted to $S_n$, the disintegration densities on the leaves of our horizontal foliation are controlled. These results are synthesized in Proposition \ref{prop:collectionFullyCrossingCurves} in \S \ref{subsec:propertiesPartitionGc}, the main result of this section. 

This last result is a primary ingredient in the proof of Theorem \ref{thm:decayCorrelations}, the proof of which will be completed in \S \ref{sec:decayCorrelations}.

%\subsection{Foliating images by horizontal curves}\label{subsec:foliateByHorizCurves}
%
%In \ref{subsubsec:constructHatGamma} we present the algorithm used to foliate the images $S_n$ by horizontal curves. What is meant precisely by a `long enough' horizontal curve is the subject of \S \ref{subsubsec:defineSigma}, where the `curve growth time' $\sigma$ is defined. The main result of \S \ref{subsec:foliateByHorizCurves} is the tail estimate on $\sigma$ given in Proposition \ref{prop:sigmaTail}, which is proved in \S\ref{subsubsec:proveSigmaTail}.

\subsection{Construction of foliations by horizontal curves}\label{subsubsec:constructHatGamma}

Let $S \subset \T^2$ be an open subset, and write $\nu_S$ for normalized Lebesgue measure on $S$. Our aim is to build a foliation of the $n$-th image $F^n(S)$ by horizontal curves with the property that for $n$ sufficiently large, `most' of the foliating curves are sufficiently long.

\subsubsection{Standing assumptions for \S \ref{sec:shapeSet}: } The parameter $\eta \in (1/2, 1)$ is fixed. The open set $S \subset \T^2$ is such that the topological boundary $\pd S = \bar S \setminus S$ is the finite union of smooth curves, and moreover, is assumed to have the following property: for any $l > 0$,
\begin{align}\label{eq:boundaryVolumeEstimate}
\nu_S \{ p \in S : d(p, \pd S) \leq l \} \leq C_S l \, ,
\end{align}
where $C_S > 0$ is a constant independent of $l$. Let us write $S_1 = S$ and $F^{n-1} S_1 = S_n$ for $n \geq 1$, noting that $\pd S_n = F^{n-1} \pd S_1$ since each $F^n$ is a diffeomorphism.

\bigskip

For $n \geq 1$, we write $\Bc_n$ for the partition of $\T^2$ into the connected components of $ B_n$ and $ B_n^c$, noting that each is a partition of $\T^2$ into vertical cylinders (sets of the form $I \times \T^1$ for a proper connected subinterval $I \subset \T^1$. We also abuse notation somewhat and write $\pd \Bc_n$ for the 
union of the boundaries of each atom of $\Bc_n$; that is, $\pd \Bc_n$ is the union of circles of the 
form $\{\hat x_n \pm 2 K_1 L_n^{-1 + \eta} \} \times \T^1$ as $\hat x_n$ varies over $\Cc_n$.

Define the sequence of partitions $\{\Pc_n\}_{n \geq 1}$ of $\T^2$ as follows:
\[
\Pc_1 = \Bc_1 \vee \{S_1, S_1^c\} \, ,
\]
and for $n \geq 2$,
\[
\Pc_n = \Bc_n \vee F_{n-1} (\Pc_{n-1}) \, .
\]
Above, $\vee$ refers to the \emph{join} of partitions. Hereafter for $q \in \T^2$, we write $\Pc_n(q)$ for the atom of $\Pc_n$ containing $q$. Again we abuse notation somewhat and write $\pd \Pc_n$ for the union over the collection of boundaries of each atom comprising $\Pc_n$.

\medskip

\noindent {\bf Additional notation: } For $q = (x,y) \in \T^2$, let us write $H_q = \T^1 \times \{y\}$ for the horizontal circle containing $q$. When $\Pc$ is a partition of $\T^2$ and $p \in \T^2$, we write $\Pc(p)$ for the atom of $\Pc$ containing $p$. We write ``$ \leq $'' for the partial order on partitions: for partitions $\Pc, \mathcal Q$, we write $\Pc \leq \mathcal Q$ if each atom in $\Pc$ is a union of $\mathcal Q$-atoms.

\subsubsection{Algorithm for foliating $S_n$ by horizontal curves}

We now define, for each $n \geq 1$, a foliation (partition) $\hat \gamma_n$ of $S_n$ by horizontal curves. 

For $n = 1$, we define $\hat \gamma_1$ to be the partition of $S_1$ consisting of atoms of the form
\[
\hat \gamma_1(p) = H_p \cap \Pc_1(p) \, 
\]
for $p \in S_1$. Clearly $\hat \gamma_1$ is a measurable partition of $S_1$, and $\hat \gamma_1 \leq \Pc_1|_{S_1}$ (here $\leq$ indicates the partial order on partitions in terms of refinement, and $\Pc_1|_{S_1}$ denotes the restriction of $\Pc_1$ to $S_1$). Inductively, assume that
$\hat \gamma_1, \cdots, \hat \gamma_n$ have been constructed, and that $\hat \gamma_n \geq \Pc_n|_{S_n}$. To define $\hat \gamma_{n+1}(p_{n+1})$ for $p_{n+1} \in S_{n+1}$, we distinguish two cases. Below we write $p_n = F_n^{-1}(p_{n+1})$.

\smallskip
\noindent {\bf Case 1: $p_n \notin  B_n$. } By construction, $\hat \gamma_n(p_n) \cap B_n = \emptyset$, and so $F_n(\hat \gamma_n(p_n))$ is a horizontal curve (Lemma \ref{lem:forwardsGT}). In preparation for the next iterate, we cut this image curve by $\Pc_{n + 1}$; that is,
\[
\hat \gamma_{n + 1}(p_{n+1}) = F_n(\hat \gamma_n(p_n)) \cap \Pc_{n + 1}(p_{n + 1}) \, .
\]
Equivalently, $\hat \gamma_{n + 1}|_{F_n(B_n^c \cap S_n)} = F_n(\hat \gamma_n \cap B_n^c) \vee \Pc_{n + 1}|_{F_n(B_n^c \cap S_n)}$

\medskip

\noindent {\bf Case 2: $p_n \in  B_n$. } In this case $\hat \gamma_n(p_n) \subset B_n$ and so we lose our control on the image curve $F_n(\hat \gamma_n(p_n))$. The procedure here is to re-partition the entire image of $B_n$ by horizontal line segments cut by $\Pc_{n + 1}$, in preparation for the next iterate. Precisely, we define
\[
\hat \gamma_{n +1}(p_{n+1}) = H_{p_{n + 1}} \cap \Pc_{n + 1}(p_{n + 1}) \, .
\]
Equivalently, $\hat \gamma_{n + 1}|_{F_n(B_n \cap S_n)}$ is the join of $\Pc_{n + 1}|_{F_n(B_n \cap S_n)}$ with the partition of $F_n(B_n)$ into horizontal circles (sets of the form $\T^1 \times \{ y \} \subset \T^2$ for $y \in \T^1$).

\bigskip

This induction procedure bootstraps because $\hat \gamma_{n + 1}$ is a partition of $S_{n + 1}$ into horizontal curves for which $\hat \gamma_{n + 1} \geq \Pc_{n + 1}|_{S_{n +1}}$. All partitions mentioned are measurable \cite{rokhlin1949fundamental}, and so we have the following.

\begin{lem}
For each $n \geq 1$, the partition $\hat \gamma_n$ of $S_n$ as above is defined and is a measurable partition of $S_n$ into connected, smooth horizontal curves for which $\hat \gamma_n \geq \Pc_n|_{S_n}$.

%The collection $\hat \Xi_n := \{\hat \gamma_n(p) : p \in S_1\}$ forms a partition of $S_n$ into
%horizontal curves, modulo a zero Lebesgue measure set.
\end{lem}
%{\color{red} General comment: what am I doing with boundaries???}

\subsection{Estimating time to curve length growth}\label{subsubsec:defineSigma}

As indicated in the procedure laid out above, the curves of $\hat \gamma_{n + 1}$ coming from $\hat \gamma_n|_{S_n \cap B_n^c}$ have been elongated by the strong expansion of $F_n$ along horizontal directions. However, this elongation of curves competes with the `cutting' of curves near bad sets (case 1) and the occasional `repartitioning' of the images of the bad sets $S_n \cap B_n$ by horizontal line segments (case 2). Our aim now is to show that for large $n$, the expansion wins out, and `most' of the curves comprising the foliation $\hat \gamma_n$ are of sufficiently long horizontal extent.

\newcommand{\Rad}{\operatorname{Rad}}

\subsubsection{Preparations} \label{subsubsec:preparationsShapeSet}
For a connected $C^1$ curve $\gamma \subset \T^2$ and a point $q = (x,y) \in \gamma$, we define 
\[
\Rad_q(\gamma) = d_\gamma(q, \pd \gamma) \, ;
\]
Here $d_\gamma$ denotes the Euclidean distance on $\gamma$, and $\pd \gamma$ denotes the endpoints of $\gamma$; that is, if $\gamma = \graph h_\gamma$ for $h_\gamma : I_\gamma \to \T^1$, then $\pd \gamma = \{ (\hat x, h_\gamma(\hat x)) : \hat x \in \pd I_\gamma\}.$ Recall that $I_\gamma \subset \T^1$ is always a proper connected subarc, so $\pd I_\gamma$, hence $\pd \gamma$, consists of exactly two points.

Additionally, let us define the following alternative of the time $\tau$ defined in \S \ref{subsec:basicConstruct}: for $p \in \T^2$, we define
\begin{align*}
\bar \tau(p) & = 1 + \max \{ m \geq 1 : d(F^{m-1}(x,y), B_m) < K_1 L_m^{-1 + \eta'} \} \\
& = \min\{ k \geq 1 : d(F^{n-1} (p), B_n) \geq K_1 L_n^{-1 + \eta'} \,\, \text{ for all } n \geq k \} \, .
\end{align*}
Here, we have set
\[
\eta' = \frac{\eta + 1}{2} \, .
\]
Clearly $\tau \leq \bar \tau$. A straightforward variation of the argument for Lemma \ref{lem:borelCantelli} implies that $\bar \tau$ is almost surely finite and satisfies an analogous tail estimate to that of $\tau$ whenever $\sum_n L_n^{-1 + \eta'} < \infty$. Precisely, we have
\begin{align}\label{eq:tailBarTau}
\Leb \{ \bar \tau > N \} \leq \sum_{n = N}^\infty 6 K_1 L_n^{-1 + \eta'} = O \bigg( \sum_{n \geq N} L_n^{-1 + \eta'} \bigg) \, .
\end{align}

\medskip

\noindent {\it For the remainder of Section \ref{sec:shapeSet}, we shall assume that the sequence $\{ L_n \}$ is such that the right-hand side of \eqref{eq:tailBarTau} is finite.
}

\subsubsection{The curve growth time $\sigma_S$}

\begin{defn}
Given $p \in S_1$, we define the \emph{curve growth time} $\sigma_S(p)$ by
\[
\sigma_S(p) = \min\{ k \geq \bar \tau(p) : \Rad_{p_k}(\hat \gamma_k(p_k)) \geq K_1 L^{-1 + \eta'}_k \} \, ,
\]
where above we write $p_k = F^{k-1} (p)$.
%{ \color{red} definition of the curve growth time $\sigma(p)$ for generic $p \in R_1$}
\end{defn}
\noindent In this section, we write $\sigma = \sigma_S$ for short. 

Our definition of $\sigma$ is motivated by the following consideration. Let $p \in S_1, p_n = F^{n-1}(p)$, and assume $\sigma(p) = n$. Then, $\hat \gamma_n(p_n) \cap B_n = \emptyset$, and $|I_{\hat \gamma_n(p_n)}| \geq 2 K_1 L^{-1 + \eta'}_n$: this implies that $F_n(\hat \gamma_n(p_n))$ is a union of approximately $L_n^{2 \eta - 1} \gg 1$ fully crosssing horizontal curves. Thus $\sigma$ has the connotation of a \emph{mixing time}: the set $\{ \sigma \leq n \} \subset S$ is a region of $S$ which has proliferated throughout $\T^2$.

\medskip

A possible obstruction to mixing is that once this mass has proliferated, it could become `trapped' again by the bad sets $B_n$. This it not possible, however, due to the way that $\sigma$ is defined. Precisely, we have the following.
\begin{lem}
Let $p \in S$, and assume that $\sigma(p) = n$ for some $n \geq 1$. Then, $\Rad_{p_k} (\hat \gamma_k(p_k)) \geq K_1 L_k^{-1 + \eta'}$ for all $k \geq n$. 
\end{lem}

\begin{proof}
It suffices to show that for any $k \geq \bar \tau(p)$, we have that $\Rad_{p_k}(\hat \gamma_k(p_k)) \geq K_1 L_k^{-1 + \eta'}$ implies $\Rad_{p_{k + 1}}(\hat \gamma_k(p_k)) \geq K_1 L_k^{-1 + \eta'}$. This is implied directly by Lemma \ref{lem:multMixingPrep}.
\end{proof}

The main result of \S \ref{subsubsec:defineSigma} is the following estimate on the tail of $\sigma$:

\begin{prop}\label{prop:sigmaTail}
There is a constant $C$, depending only on $K_1, M_0$, such that the following holds. Let $L_0$ be sufficiently large. Then, for any $n \geq 1$, we have that
\[
\nu\{ \sigma(p) > 4n \} \leq \bigg(\frac{C}{\Leb(S)} + C_S \bigg) \sum_{i = n}^\infty L^{-1 + \eta'}_i \, .
\]
%\[
%\Leb\{ \sigma(p) > n\} \leq 4 K_1 M_0 \sum_{i = N}^\infty L^{-1 + \eta}_i + 
%\big( 2 N M_0 + 4 l \big) \frac{\prod_{i = 1}^{N-1} 2 K_0 L_i}{\prod_{i = N}^{n-1} L_i^{\eta'} } \, .
%\]
\end{prop}
%\noindent The proof occupies the remainder of \S \ref{subsec:foliateByHorizCurves}.
\noindent Proposition \ref{prop:sigmaTail} bears a strong resemblance to the Volume Lemma in billiard dynamics, used to control the lengths of unstable manifolds; see, e.g., \cite{chernov2006chaotic}.

\begin{rmk}
Let us draw a comparison between the present situation and that of a typical nonuniformly hyperbolic system for which correlation decay and statistical properties are known, e.g., systems admitting Young towers with controllable `good' return times to its base \cite{young1998statistical}. Roughly speaking, the typical situation is that a given `lump' of mass can fail to proliferate: for example, nice hyperbolic geometry can be spoiled (as happens for Henon maps; see, e.g., \cite{benedicks1991dynamics}), or mass may become `trapped' somewhere (as happens for intermittent maps; see, e.g., \cite{liverani1999probabilistic}). In a typical situation admitting a Young tower, a given `lump' of mass experiences infinitely many `proliferations' (returns to the base), followed by some possibly unbounded `reset' time (sojourn up the tower) before the next proliferation takes place. Thus, correlation decay estimates depend critically on the delicate balance between these two behaviors.

In contrast, the situation for our composition $\{ F^n \}$ is simpler: at any time, some positive proportion of $\nu_n$ is `trapped' in a bad region, but as time evolves, an increasingly larger proportion of the mass of $\nu_n$ has `permanently proliferated' throughout $\T^2$. 
\end{rmk}

\subsubsection{Proof of Proposition \ref{prop:sigmaTail}}
We require two estimates:
\begin{itemize}
\item[(A)] for any $p_n \in S_n, n \geq 1$, a `bad' a priori estimate on $\Rad_{p_n} (\hat \gamma_n(p_n))$; and
\item[(B)] for $\Leb$-almost every $p \in S_1$, a `good' estimate for $\Rad_{p_n}(\hat \gamma_n(p_n))$ for $n \gg \bar \tau(p)$ (where $p_n = F^{n-1} (p)$.
\end{itemize}
\noindent Afterwards, we will (C) synthesize these estimates to obtain the desired estimate on the tail of $\sigma$.

\medskip

Let us briefly elaborate on this strategy. Before time $\bar \tau(p)$, we have no control whatsoever on the orbit of $p$, and so our procedure may indeed produce very short curves $\hat \gamma_n(p_n), p_n = F^{n-1}(p)$ for such $n$. As a result, we have access to only the `worst possible' estimates for $\Rad_{p_n}(\hat \gamma_n(p_n))$. We carry these estimates out in (A) below.  Once $\bar \tau(p)$ has elapsed, we will leverage our control on the orbit of $p$ after time $\bar \tau(p)$ to grow the curves $\hat \gamma_n(p_n)$ to sufficient horizontal extent-- this is carried out in part (B).

\subsubsection*{(A) `Bad' a priori length estimate for $\hat \gamma_n(p_n)$ for all $n$}

Here we prove the following estimate.

\begin{lem}\label{lem:badAPrioriEst}
Let $p_1 \in S_1$ and write $p_k = F^{k-1} p_1$ for $k > 1$. Then, for any $n \geq 1$,
\[
\Rad_{p_n}(\hat \gamma_n(p_n)) \geq \min\bigg\{ \min_{1 \leq i \leq n }\bigg\{ \bigg(\prod_{j = i}^{n-1} 2 K_0 L_{j} \bigg)^{-1} d(p_{i}, \pd \Bc_{i}) \bigg\},  \bigg(\prod_{j = 1}^{n-1} 2 K_0 L_{j}\bigg) ^{-1} d(p_{1}, \pd S_{1}) \bigg\} \, .
\]
\end{lem}
%\noindent In Lemma \ref{lem:badAPrioriEst} and below, the empty product $\prod_{j = n}^{n-1}$ is to interpreted as equal to $1$.

\noindent Lemma \ref{lem:badAPrioriEst} will be obtained from the corresponding identical estimate for $d(p_n, \pd \Pc_n)$.
\begin{lem}\label{lem:boundaryEstimate}
In the setting of Lemma \ref{lem:badAPrioriEst}, we have
\[
d(p_n, \pd \Pc_n) \geq \min\bigg\{ \min_{1 \leq i \leq n }\bigg\{ \bigg(\prod_{j = i}^{n-1} 2 K_0 L_{j} \bigg)^{-1} d(p_{i}, \pd \Bc_{i}) \bigg\},  \bigg(\prod_{j = 1}^{n-1} 2 K_0 L_{j}\bigg) ^{-1} d(p_{1}, \pd S_{1}) \bigg\} \, .
\]
\end{lem}

\noindent In both of Lemmas \ref{lem:badAPrioriEst} and \ref{lem:boundaryEstimate}, the empty product $\prod_{j = n}^{n-1}$ is to interpreted as equal to $1$.

\begin{proof}[Proof of Lemma \ref{lem:boundaryEstimate}]
%In the setting of Lemma \ref{lem:badAPrioriEst}, it suffices to prove that
%\[
%d(p_n, \pd \Pc_n) \geq \min\bigg\{ \min_{1 \leq i \leq n }\bigg\{ \bigg(\prod_{j = i}^{n-1} 2 K_0 L_{j} \bigg)^{-1} d(p_{i}, \pd \Bc_{i}) \bigg\},  \bigg(\prod_{j = 1}^{n-1} 2 K_0 L_{j}\bigg) ^{-1} d(p_{1}, \pd S_{1}) \bigg\} \, ,
%\]
%since by construction we have that $\hat \gamma_n \geq \Pc_n|_{S_n}$.

To prove this estimate, recall that for $k \geq 1$ we have $\pd \Pc_k = \pd \Bc_k \cup F_{k-1} (\pd \Pc_{k-1})$; thus
\[
d(p_k, \pd \Pc_k) = \min\{ d(p_k, \pd \Bc_k) , d(p_k, F_{k-1}(\pd \Pc_{k-1}) \} \, .
\]
{Noting that $\Lip(F_{k-1}^{-1}) \leq 2 K_0 L_{k-1}$, we obtain}
\[
d(p_k, F_{k-1}(\pd \Pc_{k-1}) ) \geq (2 K_0 L_{k-1})^{-1} d(p_{k-1}, \pd \Pc_{k-1}) \, .
\]
Thus for all $n \geq 2$ we obtain the following. Below we write {$a \wedge b = \min\{a, b\}$} for short.
%\begin{align*}
%d(p_k, \pd \Pc_k) & \geq \min\{ d(p_k, \pd \Bc_k) , (2 K_0 L_{k-1})^{-1} d(p_{k-1}, \pd \Pc_{k-1}) \} \\
%& \geq \min\{ d(p_k, \pd \Bc_k) , (2 K_0 L_{k-1})^{-1} d(p_{k-1}, \pd \Bc_{k-1}), (2 K_0 L_{k-1})^{-1} (2 K_0 L_{k-2})^{-1} d(p_{k-2}, \pd \Pc_{k-2}) \} \\
%& \geq \cdots \geq d(p_k, \pd \Bc_k) \wedge \min_{1 \leq j \leq i }\bigg\{ \bigg(\prod_{l = 1}^j 2 K_0 L_{k - l} \bigg)^{-1} d(p_{k - j}, \pd \Bc_{k-j}) \bigg\} \\
%& \wedge \bigg(\prod_{j = 1}^{i + 1} 2 K_0 L_{k - j}\bigg) ^{-1} d(p_{k - (i + 1)}, \pd \Pc_{k - (i + 1)})  \, .
%\end{align*}
\begin{align*}
d(p_n, \pd \Pc_n) & \geq \min\{ d(p_n, \pd \Bc_n) , (2 K_0 L_{n-1})^{-1} d(p_{n-1}, \pd \Pc_{n-1}) \} \\
& \geq \min\{ d(p_n, \pd \Bc_n) , (2 K_0 L_{n-1})^{-1} d(p_{n-1}, \pd \Bc_{n-1}), (2 K_0 L_{n-1})^{-1} (2 K_0 L_{n-2})^{-1} d(p_{n-2}, \pd \Pc_{n-2}) \} \\
& \geq \cdots \geq d(p_n, \pd \Bc_n) \wedge \min_{2 \leq i \leq n-1 }\bigg\{ \bigg(\prod_{j = i}^{n-1} 2 K_0 L_{j} \bigg)^{-1} d(p_{i}, \pd \Bc_{i}) \bigg\} \wedge \bigg(\prod_{j = 1}^{n-1} 2 K_0 L_{j}\bigg) ^{-1} d(p_{1}, \pd \Pc_{1})  \, .
\end{align*}
The desired estimate now follows from the fact that $\pd \Pc_1 = \pd S_1 \cup \pd \Bc_1$.
\end{proof}

\begin{proof}[Proof of Lemma \ref{lem:badAPrioriEst}]
With $n \in \N$ fixed, define 
\[
n_1 = \max \{ 1 \leq k \leq n - 1 : p_k \in B_k\} \, ,
\]
where we use the \emph{ad hoc} convention $n_1 = 1$ if $p_k \notin B_k$ for all $1 \leq k \leq n-1$. Observe that
$\hat \gamma_{n_1 + 1}(p_{n_1 + 1})$ is formed by using Case 2 in the algorithm, and that $\hat \gamma_{k}(p_k)$ is formed using
Case 1 for every $k \geq n_1 + 2$. In particular,
\[
\Rad_{p_{n_1 + 1}}(\hat \gamma_{n_1 + 1}(p_{n_1 + 1})) \geq d(p_{n_1 + 1}, \pd \Pc_{n_1 + 1}) \, ,
\]
and for every $n_1 + 2 \leq k \leq n$, we have
\[
\Rad_{p_k}(\hat \gamma_k(p_k)) \geq \min\{ d(p_k, \pd \Pc_k), \Rad_{p_k}(F_{k-1}( \hat \gamma_{k -1}(p_{k-1}) ) \} \, .
%\Rad_{p_n} (\hat \gamma_n(p)) \geq d(p_n, \pd \Pc_n) 
\]

To prove Lemma \ref{lem:badAPrioriEst} it suffices to show that
\begin{align}\label{eq:nTon1Estimate}
\Rad_{p_n}( \hat \gamma_n(p_n)) \geq \min_{ n_1 + 1 \leq k \leq n} \{ d(p_k, \pd \Pc_k) \} \, . 
\end{align}
Once \eqref{eq:nTon1Estimate} is proved, Lemma \ref{lem:badAPrioriEst} follows on plugging in the estimates
for $d(p_k, \pd \Pc_k)$ for $1 \leq k \leq n$.

Turning to \eqref{eq:nTon1Estimate}: if $n_1 = n-1$ then there is nothing left to show. If $n_1 < n-1$, then we estimate:
\begin{align*}
\Rad_{p_n}(\hat \gamma_n) & \geq d(p_n, \pd \Pc_n) \wedge  L^{\eta}_{n-1} \Rad_{p_{n-1}}(\hat \gamma_{n-1}(p)) \, ,\\
& \geq d(p_n, \pd \Pc_n) \wedge  L^{\eta}_{n-1} d(p_{n-1}, \pd \Pc_{n-1}) \wedge  L^{\eta}_{n-1}  L^{\eta'}_{n-2}
\Rad_{p_{n-2}}(\hat \gamma_{n-2}) \geq \cdots \\
& \geq d(p_n, \pd \Pc_n) \wedge \min_{n_1 + 2 \leq i \leq n-1} \bigg\{ \bigg( \prod_{j = i}^{n-1}  L_j^{\eta} \bigg) d(p_i, \pd \Pc_i) \bigg\} \wedge \Rad_{p_{n_1 + 1}}(\hat \gamma_{n_1 + 1}(p_{n_1 + 1})) \, .
\end{align*}
Here we have used the simple estimate
\begin{align}\label{eq:expansionPartitionAtoms}
\Rad_{p_{j+1}}F_j (\hat \gamma_j(p_j)) \geq  L_j^\eta \Rad_{p_j}(\hat \gamma_j(p_j)) \, ,
\end{align}
which follows from the expansion estimate along horizontal curves in Lemma \ref{lem:forwardsGT}. Replacing all $ L_j^{\eta}$ terms with $1$, we obtain \eqref{eq:nTon1Estimate}. 
\end{proof}

\subsubsection*{(B) Good length estimate for $\hat \gamma_n(p_n)$ for $n \gg \tau(p)$}

Here we prove the following.

\begin{lem}\label{lem:goodCurveEstimate}
Let $N \geq 1$, and let $p \in S_1$ be such that $\bar \tau(p) \leq N < \infty$. Then for any $n \geq N$,
\begin{align}
\Rad_{p_n}(\hat \gamma_n(p_n)) \geq \min\bigg\{ d(p_n , \pd \Bc_n) , \bigg(\prod_{k = N}^{n-1}  L_k^{\eta}\bigg) \Rad_{p_N}(\hat \gamma_N(p_N)))\bigg\} \, .
\end{align}
\end{lem}

\begin{proof}[Proof of Lemma \ref{lem:goodCurveEstimate}]
The proof leans on the following claim.

\begin{cla}
Let $p \in S_1$ be such that $\bar \tau(p) \leq N <\infty$. Then, for all $n \geq N$, we have
\begin{align*}
\Rad_{p_{n + 1}}(\hat \gamma_{n + 1}(p_{n+1})) \geq \min\{d(p_{n + 1}, \pd \Bc_{n + 1}) , \Rad_{p_n}(F_n(\hat \gamma_{n}(p_n))) \} \, .
\end{align*}
\end{cla}
\begin{proof}[Proof of Claim]
Observe that since $n \geq \bar \tau(p) \geq \tau(p)$, we always use Case 1 in the construction of $\hat \gamma_{n + 1}(p_{n+1})$, i.e., $\hat \gamma_{n + 1}(p_{n+1}) = F_n(\hat \gamma_n(p_n)) \cap \pd \Pc_{n + 1}(p_{n+1})$. Moreover,
 $\hat \gamma_n(p_n) \subset \Pc_n(p_n)$ by construction, hence $F_n(\hat \gamma_n(p_n)) \subset F_n(\Pc_n(p_n))$, and
 so we arrive at
 \[
 \hat \gamma_{n + 1}(p_{n+1}) = F_n(\hat \gamma_n(p_n)) \cap \Bc_{n + 1}(p_{n + 1}) \, .
 \]
The desired estimate now follows.
\end{proof}

Fixing $n \geq N$, we now estimate
\[
\Rad_{p_n}(\hat \gamma_n(p_n)) \geq \min\{ d(p_{n}, \pd \Bc_{n}), \Rad_{p_n}(F_{n-1}(\hat \gamma_{n-1}(p_{n-1}))) \} \, .
\]
Observe that since $d(p_{n-1}, \pd \Bc_{n-1}) \geq K_1 L_{n-1}^{-1 + \eta}$, it follows that 
\[
\Rad_{p_{n }}(F_{n-1}(\hat \gamma_{{n-1}}(p_{n-1}))) \geq  L_{n-1}^{\eta} \Rad_{p_{n-1}}(\hat \gamma_{n-1}(p_{n-1})) \, .
\]
on applying \eqref{eq:expansionPartitionAtoms}. Iterating,
\begin{align*}
\Rad_{p_n}(\hat \gamma_n(p_n)) \geq d(p_n, \Bc_n) \wedge \min_{N \leq k \leq n-1}\bigg\{ \bigg(\prod_{i = k}^{n-1}  L_k^{\eta}\bigg) d (p_k, \pd \Bc_k) \bigg\} \wedge \bigg(\prod_{i = N}^{n-1}  L_i^{\eta}\bigg) \Rad_{p_N} (\hat \gamma_N(p_N)) \, .
%\geq \min\{ 1, (\prod_{k = N}^{n-1} L_k^{\eta'}) \Len(\hat \gamma_N(p)))\} \, .
\end{align*}
Note however that for $N \leq k \leq n-1$, we have that
\[
d(p_{k}, \pd \Bc_{k}) \geq K_1 L_{k}^{-1 + \eta'} \, ,
\]
hence $ L_k^{\eta} \cdot d(p_k, \pd \Bc_k) \geq K_1 L^{2 (\eta + \eta') - 1} \gg 1$ (recall $\eta > 1/2$ so $ \eta + \eta' - 1 > 2 \eta  - 1> 0$) when $L_0$ is sufficiently large in terms of $K_1, \eta$. This yields the desired estimate.
\end{proof}

\subsubsection*{(C) Final estimates on the tail of $\sigma$}

We are now in position to prove our estimate on $\Leb\{ p \in S_1 : \sigma(p) > 4 n\}$. Assume that
$p \in S_1$ and $\bar \tau(p) \leq n < \infty$; finally, assume $\sigma(p) > 4 n$. From Lemma \ref{lem:goodCurveEstimate} it follows that
\[
\Rad_{p_n}(\hat \gamma_n(p_n)) < K_1 L^{-1 + \eta'}_{4n} \cdot \bigg( \prod_{k = n}^{4n-1}  L_k^{\eta}\bigg)^{-1} \leq \bigg( \prod_{k = n}^{4n-1}  L_k^{\eta}\bigg)^{-1}  
\]
for $L_0$ sufficiently large, since here we always have $d(p_{4n}, \pd \Bc_{4n}) \geq K_1 L_{4n}^{-1 + \eta'}$ by definition of $\bar \tau, \sigma$. Plugging in our estimate from Lemma \ref{lem:badAPrioriEst}, there are two cases to consider:

\noindent {\bf Case (a): } For some $1 \leq k \leq n$, we have
\[
 d( p_k, \pd \Bc_k) < \frac{\prod_{i = k}^{n-1} 2 K_0 L_i}{\prod_{i = n}^{4n-1}  L_i^{\eta} } \, ,
\]
(again the empty product $\prod_{i = n}^{n-1}$ is taken to equal 1) or

\noindent {\bf Case (b): } we have
\[
 d( p_1, \pd S_1) < \frac{\prod_{i = 1}^{n-1} 2 K_0 L_i}{\prod_{i = n}^{4n-1}  L_i^{\eta} } \, .
\]

By volume preservation, it follows that for each $1 \leq k \leq n-1$,
\[
\Leb\bigg\{ \begin{array}{c} p \in S_1 : \tau(p) \leq n \, , \sigma(p) > 4n \, ,\\
 \text{ and Case (a) holds for value $k$} \end{array}
 \bigg\}
 \leq 
 2 \# (\Cc_k) \cdot \frac{\prod_{i = k}^{n-1} 2 K_0 L_i}{\prod_{i = n}^{4n-1}  L_i^{\eta} } \, .
\]
Additionally, using the estimate \eqref{eq:boundaryVolumeEstimate}, we have
\begin{align*}
\Leb\bigg\{ \begin{array}{c} p \in S_1 : \tau(p) \leq n \, , \sigma(p) > 4n \, ,\\
 \text{ and Case (b) holds} \end{array}
 \bigg\}
&\leq
\Leb \bigg\{ p \in S_1 : d(p, \pd S_1) \leq \frac{\prod_{i = 1}^{n-1} 2 K_0 L_i}{\prod_{i = n}^{4n-1}  L_i^{\eta} } \bigg\} \\
& \leq C_S \Leb(S) \frac{\prod_{i = 1}^{n-1} 2 K_0 L_i}{\prod_{i = n}^{4n-1}  L_i^{\eta} } \, .
\end{align*}

Thus
\begin{align}\label{eq:prelimSigmaEst}
\Leb\{ p \in S_1 : \tau(p) \leq n, \sigma(p) > 4n \} \leq \big( 2 n M_0 + C_S \Leb(S) \big) \frac{\prod_{i = 1}^{n-1} 2 K_0 L_i}{\prod_{i = n}^{4n-1}  L_i^{\eta} } \, .
\end{align}
To develop the right-hand side, observe that
\[
\frac{\prod_{i = 1}^{n-1} 2 K_0 L_i}{\prod_{i = n}^{3n-1}  L_i^{\eta}} \leq \prod_{i = 1}^n 2 K_0 L_i^{1 - 2 \eta} \leq 1 \, ,
\]
using that $\{L_i\}$ is a nondecreasing sequence, on taking $L_0$ sufficiently large so that $2 K_0 L_0^{1 - 2 \eta} \leq 1$. For the terms $i = 3n, \cdots, 4n-1$, we estimate:
\[
\prod_{i = 3n }^{4n-1} 2 L_i^{- \eta} 
= \bigg( \prod_{i =3n}^{4n-1}  L_i^{- \eta n}  \bigg)^{1/n} \leq \frac1n \sum_{i = 3n}^{4n - 1}  L_i^{- \eta n} 
\leq \frac1n \sum_{i = 3n}^{4n - 1} L_i^{-1 + \eta}
\]
by AMGM, on noting $ L_n^{- \eta n} < L_n^{-1 + \eta}$ for all $n \geq 1$. Thus
\[
\eqref{eq:prelimSigmaEst} \leq (2 M_0 + C_S \Leb(S)) \sum_{i = 3n}^{4n - 1} L_i^{-1 +\eta} \, .
\]

For the final estimate, observe that
\begin{align*}
\Leb\{ p \in S_1 : \sigma(p) > 4 n\} & \leq \Leb\{ p \in S_1 : \bar \tau(p) \leq n, \sigma(p) > 4 n\} + \Leb\{ p \in S_1 : \bar \tau(p) > n\} \\ 
& \leq (2 M_0 + C_S) \sum_{i = 3n}^{4n - 1} L_i^{-1 + \eta} + 6 K_1 M_0 \sum_{i = n}^\infty L_i^{-1 + \eta'} \\
& \leq 
\big( 2 M_0 + C_S \Leb(S) + 6 K_1 M_0 \big) \sum_{i = n}^\infty L_i^{-1 +  \eta'} 
\end{align*}
on using \eqref{eq:tailBarTau} and that $\{L_i\}$ is nondecreasing. This completes the proof of Proposition \ref{prop:sigmaTail}.

\subsection{Disintegration of Lebesgue measure along horizontal foliation $\hat \gamma_n$}\label{subsubsec:disintegrationNuN}

To complete our description of the foliation $\hat \gamma_n$ of $S_n$, we describe here how $\hat \gamma_n$ disintegrates Lebesgue measure $\nu_n=  F^{n-1}_* \nu_S = \frac{1}{\Leb_{\T^2}(S_n)} \Leb_{\T^2}|_{S_n}$ on $S_n$.
 
%Above we gave an algorithm for generating horizontal curves $\hat \Xi_n = \{\hat \gamma_n\}$. Here we give an accompanying construction of measures on the individual curves $\hat \gamma_n$, and show how this construction gives rise to the canonical disintegration of $\nu_n$ with respect to $\hat \Xi_n = \{\hat \gamma_n\}$ regarded as a measurable partition of $S_n$.
%
%We construct the disintegration measures pointwise: for now, fix $p = p_1 \in S_1$ and write $p_k := F^{k-1} p$ for $k \geq 2$.  
Below, for $n \geq 1$ and an atom $\gamma \in \hat \gamma_n$, we write $(\nu_n)_\gamma$ for the disintegration measure of $\nu_n$ on $\gamma$; the disintegration measures $(\nu_n)_\gamma$ are the (almost surely) unique family of probability measures, supported on the $\gamma \in \hat \gamma_n$, which satisfy
\[
\nu_n(K) = \int_{\gamma \in S_n / \hat \gamma_n} (\nu_n)_\gamma(\gamma \cap K) \, d \nu_n^T(\gamma) 
\]
for Borel $K \subset \T^2$; here $\nu_n^T$ is the pushforward of $\nu_n$ onto the quotient space of equivalence classes $S_n / \hat \gamma_n$.

\begin{lem}\label{lem:disintegrationDensityDescription}
Let $n \geq 1$ and fix $\gamma \in \hat \gamma_n$. Let $\rho^n_\gamma$ denote the density of $(\nu_n)_\gamma$ with respect to $\Leb_{\gamma}$. Then, for any $p, q \in \gamma$, we have that
\[
\frac{\rho^n_\gamma(p)}{\rho^n_\gamma(q)} = \frac{\det (dF^{n-1}_{n_1 + 1}|_{T \gamma_{n_1+1}}) \circ (F^{n-1}_{n_1 + 1})^{-1}(q) }{\det (dF^{n-1}_{n_1 + 1}|_{T \gamma_{n_1+1}}) \circ (F^{n-1}_{n_1 + 1})^{-1}(p)}
\]
Here $n_1 =   \max \big( \{ 0 \} \cup \{ 1 \leq k \leq n -1 : p_k \in B_k \} \big) $, $p_n \in \gamma$ is an (arbitrary) representative and $p_k \in S_k$ is such that $F_k^{n-1} p_k = p_n$ for each $k \leq n$, and $\gamma_{n_1+1}$ is the atom in $\hat \gamma_{n_1+1}$ for which $F^{n-1}_{n_1+1}(\gamma_{n_1+1}) \supset \gamma$.
\end{lem}
\begin{proof}

To start let us describe the disintegration measures $(\nu_1)_{\hat \gamma_1(p_1)}$ for $p_1 \in S_1$. It is clear that
\begin{align}\label{eq:disintegrationMeasTime1}
(\nu_1)_{\hat \gamma_1(p)} =  \frac{1}{\Leb_{H_{p_1}} (\hat \gamma_1(p_1))}\Leb_{H_{p_1}}|_{\hat \gamma_1(p_1)} \, ,
\end{align}
where $H_{p_1}$ is as in \S  \ref{subsubsec:constructHatGamma} and $\Len(\gamma)$ denotes the arc length of a smooth connected curve $\gamma \subset \T^2$. Thus Lemma \ref{lem:disintegrationDensityDescription} holds trivially in this case with $n_1 = 1$.

Inductively, let us express the disintegration $\nu_{n+1}$ in terms of that for $\nu_n$. Observe that
\[
\nu_{n+1} = (F_n)_* \nu_n|_{S_n \cap B_n} + (F_n)_* \nu_n|_{S_n \setminus B_n} \, ;
\]
since $S_n \cap B_n, S_n \setminus B_n \in \Pc_n$ it suffices to consider these separately in working out the disintegration measures $(\nu_{n+1})_{\gamma}, \gamma \in \hat \gamma_{n+1}$.

On $F_n(S_n \cap B_n)$, Case 2 is applied in constructing $\hat \gamma_{n+1}|_{F_n(S_n \cap B_n)}$, and so disintegration measures are obtained using the analogue of \eqref{eq:disintegrationMeasTime1} with $n+1$ replacing $1$. 

On $F_n(S_n \setminus B_n)$, we apply Case 1 in the construction of $\hat \gamma_{n+1}$, i.e., $\hat \gamma_{n+1} = \Pc_{n+1} |_{F_n(S_n \cap B_n)} \vee F_n(\hat \gamma_n|_{S_n \cap B_n})$. In particular, the disintegration $(\nu_{n+1}|_{F_n(S_n \setminus B_n)})_\gamma, \gamma \in \hat \gamma_{n+1}$ can be obtained by disintegrating, for each $\check \gamma \in \hat \gamma_n$, the measures $(F_n)_* \big( (\nu_n)_{\check \gamma} \big)$ against the (finite) partition $\Pc_{n+1}|_{F_n(\check \gamma)}$. To wit, if $\gamma \in \hat \gamma_{n+1}|_{F_n(S_n \setminus B_n)}$ has $\gamma \subset F_n(\check \gamma)$ for $\check \gamma \in \hat \gamma_{n+1}$, then
\begin{align*}
(\nu_{n+1})_\gamma &= \frac{1}{(\nu_n)_{\check \gamma} (F_n^{-1}  \gamma) } (F_n)_* \big( (\nu_n)_{\check \gamma} \big) |_{\gamma} \, .
\end{align*}
In particular, we have shown that for any $p, q \in \gamma$, we have that
\[
\frac{\rho^{n+1}_\gamma(p)}{\rho_\gamma^{n+1}(q)} = \frac{\det(d F_n|_{T \check \gamma}) \circ F_n^{-1}(q)}{\det(d F_n|_{T \check \gamma}) \circ F_n^{-1}(p)} \cdot \frac{\rho^n_{\check \gamma} \circ F_n^{-1}(p)}{\rho^n_{\check \gamma} \circ F_n^{-1} (q) } \, .
\]
Lemma \ref{lem:disintegrationDensityDescription} follows by iterating the above relations from $n_1 + 1$ to $n-1$.
\end{proof}

\subsection{Description of $(F^n)_* \nu_S$}\label{subsec:propertiesPartitionGc}

Here we synthesize the results of \S \ref{subsubsec:constructHatGamma} -- \ref{subsubsec:disintegrationNuN} into our main result, a precise description of the bulk of $(F^n)_* \nu_S$ 
as foliated by a collection of fully crossing horizontal curves with controlled disintegration densities.

\begin{prop}\label{prop:collectionFullyCrossingCurves}
Let $n \geq 2$. Then, there is a measurable set $G \subset F^n S$ and a measurable partition $\Gc$ of $G$ with the following properties.
\begin{itemize}
\item[(a)] Each atom $\gamma \in \Gc$ is of the form $\graph h_\gamma$ where $h_\gamma : (0,1) \to \T^1$ is a $C^2$, fully crossing horizontal curve with $\| h_\gamma'\|_{C^0} = O(L_n^{- \eta})$.
\item[(b)] We have the estimate 
\begin{align}\label{eq:nuNestimateG}\begin{split}
\nu_{n+1}(G) & \geq 1 - O(L_n^{- \frac12(1 - \eta)}) - \nu_S \{ \sigma > n \}\\
& \geq 1 - \bigg(O(1) + C_S + \frac{C }{\Leb(S)} \bigg) \sum_{i = \lfloor n/4 \rfloor }^\infty L^{-\frac12(1 - \eta)}_i 
\end{split} \end{align}
on plugging in the estimate in Proposition \ref{prop:sigmaTail}.
\item[(c)] Let $\nu_G$ denote the restriction $\nu_{n+1}|_{G}$ and let $\{(\nu_G)_\gamma \}_{\gamma \in \Gc}$ denote the canonical disintegration of $\nu_G$ with respect to $\Gc$ by probability measures supported on each $\gamma \in \Gc$. Let $\rho_\gamma : \gamma \to [0,\infty)$ denote the density of $(\nu_G)_\gamma$ with respect to $\Leb_\gamma$. Then for any $p_1, p_2 \in \gamma$ we have
\[
\frac{\rho(p_1)}{\rho(p_2)} \leq e^{C L_n^{1 - 2 \eta}} \, .
\]
\end{itemize}
\end{prop}
%\noindent During the remainder of \S \ref{subsec:propertiesPartitionGc} we prove this Proposition.

\begin{proof}
To start, define
\[
\hat \Gc_n = \{ \hat \gamma \in \hat \gamma_n :  (\nu_n)_{\hat   \gamma} F^{n-1} \{ \sigma \leq n \} > 0\}  \, 
\quad \text{and} \quad
\hat G_n = \bigcup_{\hat \gamma \in \hat \Gc_n} \hat \gamma \, .
\]
By Lemma \ref{lem:partitionSaturation} in the appendix, we have
\[
\nu_n( \hat G_n) = \nu_n^T \{ \hat \gamma \in \hat \Gc_n\} \geq  \nu_{1} \{ \sigma \leq n\} \, .
\]

Recalling the notation in Lemma \ref{lem:multMixingPrep}, we define $\Gc$ by
\[
\Gc = \bigcup_{\hat \gamma \in \hat \Gc_n} \bar \Gamma_n(\hat \gamma) \, \quad \text{and} \quad
G = \bigcup_{\gamma \in \Gc} \gamma = \bigcup_{\hat \gamma \in \hat \Gc_n} F_n(\hat \gamma \setminus \Bc_n(\hat \gamma)) \, ,
\]
noting that $\Gc$ partitions $G$ into horizontal curves $\gamma$ which satisfy item (a) by construction.

\medskip

To check item (b), for each $\hat \gamma \in \hat \gamma_n$ and subset $K \subset \hat \gamma$ we have that
\[
(\nu_n)_{\hat \gamma} (K) \leq \frac{C}{\Len(\hat \gamma)} \,  \Leb_{\hat \gamma}(K) \, ,
\]
on applying the distortion estimate in Lemma \ref{lem:distControlHorizontalCurves} to the density $\rho^n_{\hat \gamma}$ derived in Lemma \ref{lem:disintegrationDensityDescription}. Since $\Len(\hat \gamma)^{-1} = O(L_n^{1 - \eta'})$ from the fact that $\hat \gamma \cap F^{n-1} \{ \sigma \leq n \} \neq \emptyset$, we obtain the estimate
\[
(\nu_n)_{\hat \gamma} (\Bc_n(\hat \gamma)) = O\bigg( \frac{L_n^{-1+\eta}}{L_n^{-1 + \eta'}} \bigg) = O(L_n^{- \frac12 (1 - \eta)})
\]
on plugging in $K = \Bc_n(\hat \gamma)$. Thus \eqref{eq:nuNestimateG} follows on noting $-1 + \eta' = -1 + \frac{1 + \eta}{2} = \frac{\eta - 1}{2}$.

\medskip

For item (c), let $p_1, p_2 \in \gamma$ for some $\gamma \in \Gc$, and assume that $\gamma \in \bar \Gamma_n(\hat \gamma)$ for $\hat \gamma \in \hat \gamma_n$. Then,
\[
\frac{\rho_\gamma(p_1)}{\rho_\gamma(p_2)} = \frac{\det (dF_n|_{T \hat \gamma}) \circ F_n^{-1}(p_2)}{\det (dF_n|_{T \hat \gamma}) \circ F_n^{-1}(p_1)} \cdot \frac{\rho^n_{\hat \gamma} \circ F_n^{-1}(p_1)}{\rho^n_{\hat \gamma} \circ F_n^{-1} (p_2)} 
\]
in the notation of \S \ref{subsubsec:disintegrationNuN}. The first factor is bounded $\leq e^{C L_n^{1 - 2 \eta} \| p_1 - p_2\|}$ by Lemma \ref{lem:distControlHorizontalCurves}. For the second factor, note that $\| F_n^{-1}(p_1) - F_n^{-1} (p_2)\| \leq L_n^{- \eta} \| p_1 - p_2 \|$ by Lemma \ref{lem:forwardsGT}, and so Lemma \ref{lem:disintegrationDensityDescription} yields the estimate $\leq e^{C L_{n_1}^{1 - 2 \eta} \cdot L_n^{- \eta} \| p_1 - p_2 \|} \leq e^{C L_n^{- \eta}}$. The estimate in item (c) follows.
\end{proof}

\section{Decay of correlations estimates}\label{sec:decayCorrelations}

\renewcommand{\phi}{\varphi}

Leaning on the mixing mechanism explored in the previous section, we complete here the proof of Theorem \ref{thm:decayCorrelations}. 

In \S \ref{subsec:docReduction}, we will show how to reduce Theorem \ref{thm:decayCorrelations} to the case when $\phi$ is the characteristic function of a small square (Proposition \ref{prop:rectangleDecayCorrelations}). In \S \ref{subsec:completeDOC} we apply the results of \S \ref{sec:shapeSet} when $S$ is a small square and give the proof of Proposition \ref{prop:rectangleDecayCorrelations}.

\bigskip

{
\it \noindent We assume throughout \S \ref{sec:decayCorrelations} that $\eta \in (1/2, 1)$ has been fixed, and that $\{ L_n \}$ has the property that $\sum_n L_n^{-1 + \eta'} < \infty$, where $\eta' = \frac{\eta + 1}{2}$ is as in \S \ref{subsubsec:preparationsShapeSet}. These assumptions are consistent with the hypotheses of Theorem \ref{thm:decayCorrelations}.
}
\subsection{Reduction}\label{subsec:docReduction}

We will show here that to prove Theorem \ref{thm:decayCorrelations}, it suffices to prove the following. 
\begin{prop}\label{prop:rectangleDecayCorrelations}
Let $R$ be a square in $\T^2$ of side length $\ell$, and let $\nu$ denote the normalized Lebesgue measure restricted to $R$. Let $\psi : \T^2 \to \R$ be $\a$-Holder continuous. Then
%\[
%\bigg| \int \psi d F^n_* \nu - \int \psi d \Leb \bigg| \leq C \| \psi\|_\a 
%%\max \{ L_n^{1 - 2 \eta}, \ell^{-2} \sum_{i = \lfloor n / 4 \rfloor}^\infty 
%\ell^{-2} 
%\max \bigg\{ L_{\lfloor n / 2 \rfloor}^{1 -2 \eta}, \bigg(\sum_{i = \lfloor n/8 \rfloor}^\infty L_i^{-1 + \eta}\bigg)^{\min\{ \a, 1/2\}} \bigg\} \, .
%\]
\[
\bigg| \int \psi \circ F^{n} \, d \nu - \int \psi \bigg| \leq C \| \psi \|_\a  \max \bigg\{ L_{\lfloor n/2 \rfloor}^{- \min\{ 2 \eta -1, \a (1 - \eta) /(\a + 2) \}} , \ell^{-2} \sum_{i = \lfloor n/8 \rfloor}^\infty L_i^{-\frac12 (1 - \eta)} \bigg\} \, .
\]
\end{prop}

%\noindent Our first objective is to deduce Theorem \ref{thm:decayCorrelations} from Proposition \ref{prop:rectangleDecayCorrelations}.

\begin{proof}[Proof of Theorem \ref{thm:decayCorrelations} assuming Proposition \ref{prop:rectangleDecayCorrelations}]
Below, $n \geq 2$ is fixed, as are $\a$-Holder continuous $\phi, \psi : \T^2 \to \R$. Let us write $\Sc_n$ for the first element in the $\max\{ \cdots \}$ in Proposition \ref{prop:rectangleDecayCorrelations} and write $\mathcal T_n$ for the summation in the second term, so that the bound on the right-hand side reads as $\leq C \| \psi \|_\a \max \{ \Sc_n, \ell^{-2} \mathcal T_n\}$.

With $K \in \N$ to be specified later, subdivide $\T^2$ into rectangles $R_{i,j}, 1 \leq i,j \leq K$ of side length $\ell = 1/K$ each. We set
\[
 \varphi_{i,j} = \inf_{p \in R_{i,j}}  \varphi(p) \, .
\]
Define $\hat \varphi := \sum_{i,j = 1}^K  \varphi_{i,j} \chi_{R_{i,j}}$, so that
\[
\int \big(  \varphi -  \hat \phi \big) d \Leb = O( \| \varphi\|_\a \cdot \ell^{\a}) \, .
\]

Let $\nu^{i,j}$ denote normalized Lebesgue measure on $R_{i,j}$. Then
\[
\int \psi \circ F^n \cdot \phi  = \int \psi \circ F^n \cdot \big( \phi - \hat \phi \big) 
% - \sum_{i,j= 1}^{K} \ell^2 \phi_{i,j} \nu^{i,j} \big) 
 +  \sum_{i,j = 1}^K \ell^{2} \phi_{i,j}  \int \psi \, d F^n_* \nu^{i,j}  \, . %= \sum_j |\nu^{(j)}|_{tot} \cdot \int \psi d F^n_* \hat \nu^{(j)} \, ,
\]
For the first term,
\[
\int \psi \circ F^n \cdot (\phi - \hat \phi) = O( \| \psi \|_\a \|\phi\|_\a \ell^\a) \, .
\]
Similarly, we estimate
\begin{align*}
\int \psi \cdot  \int \phi & = \int (\phi - \phi_k)  \cdot  \int \psi  + \sum_{i,j = 1}^{K} \ell^2 \phi_{i,j} \int \psi  \\
&= O( \| \psi \|_\a \|\phi\|_\a \ell^\a) + \sum_{i,j = 1}^K \ell^2 \phi_{i,j} \int \psi 
\end{align*}
hence
\begin{align*}
\bigg| \int \psi \circ F^n \cdot \psi  - \int \psi \int \phi \bigg| & \leq \sum_{i,j = 1}^{K} \ell^2 \phi_{i,j} \bigg| \int \psi \, d F^n_* \nu^{i,j} - \int \psi \bigg| \\
& +O( \| \psi \|_{C^0} [\phi]_\a \ell^\a) \\
& =  \| \psi \|_\a \| \phi \|_\a \cdot O\big( \Sc_n + \ell^{-2}  \mathcal T_n +   [\phi]_\a \ell^\a \big) \, .
%
%\leq 2 \| \psi\|_{C^0} [\varphi]_\a k^{- \a} \ell_k^\a + \sum_j |\nu^{(j)}|_{tot} \bigg| \int \psi d F^n_* \hat \nu^{(j)} - \int \psi d \Leb \bigg| \\
%& \leq Const. [\phi]_\a \| \psi\|_{C^0} \ell_k^\a + Const \ell^{-2}_k \| \psi\|_\a \sum_{i = n}^\infty L_i^{-1 + \eta} \, .
\end{align*}
Setting
\[
K = \bigg\lfloor \bigg( \frac{ [\phi]_\a}{ \mathcal T_n }   \bigg)^{1/(2 + \a)} \bigg\rfloor 
\]
we obtain the estimate
\[
\bigg| \int \psi \circ F^n \cdot \phi  - \int \phi \int \psi  \bigg| \leq C \| \psi \|_\a \|\phi\|_\a^{\frac{4 + \a}{2 + \a}} (  \mathcal T_n^{\frac{\a}{2 + \a}} + \Sc_n) \, .
\]

The only difference between this and our desired estimate is the exponent of $\| \phi \|_\a$ on the right-hand side. To fix this, define $\check \phi = \phi / \| \phi \|_\a$ and note $\| \check \phi \|_\a = 1$; for this function we have
\[
\bigg| \int \psi \circ F^n \cdot \check \phi  - \int \check \phi \int \psi  \bigg| \leq C \| \psi \|_\a ( \mathcal T_n^{\frac{\a}{2 + \a}} + \Sc_n) \, ,
\]
and so the desired estimate follows on multiplying both sides by $\| \phi \|_\a$. To complete the proof, observe that 
\[
\max \{ \Sc_n, {\mathcal T}_n^{\frac{\a}{\a + 2}}\} \leq \max \bigg\{ L_{\lfloor n / 2 \rfloor}^{1 - 2 \eta}, \bigg( \sum_{i = \lfloor n / 8 \rfloor} L_i^{- \frac12 (1 - \eta)} \bigg)^{\frac{\a}{\a + 2}} \bigg\}
\]
since ${\mathcal T}_n^{\frac{\a}{\a + 2}}$ always dominates $L_{\lfloor n / 2 \rfloor}^{- \a(1 - \eta) / (\a + 2)}$.
\end{proof}

%{\color{red} Standing notational kerfuffle: $\sigma$ defined at time 1 but below it's expedient to push it forward to time $n, n + 1$, etc...}

\subsection{Proof of Proposition \ref{prop:rectangleDecayCorrelations}}\label{subsec:completeDOC}

To complete the proof of Theorem \ref{thm:decayCorrelations}, it remains to prove Proposition \ref{prop:rectangleDecayCorrelations}. We combine the description in Proposition \ref{prop:collectionFullyCrossingCurves} of the foliation by long horizontal curves with the mixing estimate in Proposition \ref{prop:CLTdocCurves} along those horizontal curves.

To wit: let $\psi : \T^2 \to \R$ be $\a$-Holder continuous and let $R$ be a square of side length $\ell$ as in the statement of Proposition \ref{prop:rectangleDecayCorrelations}. With $\nu$ denoting the Lebesgue measure restricted to $R$, and (for notational convenience) appling the substitution $n \mapsto 2n$, we will estimate
\begin{align}\label{eq:2nMixEst}
\int \psi \circ F^{2n} \, d \nu = \int \psi \circ F^{2n}_{n+1} \, d (F^n_* \nu) \, .
\end{align}

For each $k \geq 1$ define $\nu_k = F^{k-1}_* \nu_1$, where $\nu_1 = \nu$. Applying Proposition \ref{prop:collectionFullyCrossingCurves} to $S = R$, we obtain the collection $\Gc$ of horizontal curves foliating the set $G \subset F^n R$. In the notation of Proposition \ref{prop:sigmaTail}, we have $C_R = O(\ell^{-1})$, and so 
\[
\nu_{n+1}(G^c) = %O(L_n^{1-2\eta}) + 
O \bigg( \ell^{-2} \sum_{i = \lfloor n/4 \rfloor}^\infty L_i^{-\frac12(1 + \eta)} \bigg) \, .
\]

Returning to the estimate of \eqref{eq:2nMixEst},
\begin{align*}
\eqref{eq:2nMixEst} & = O( \| \psi\|_\a \, \nu_{n+1}(G^c) ) + \int \psi \circ F^{2n}_{n+1} \, d \nu_G \\
& = O( \| \psi\|_\a \, \nu_{n+1}(G^c) ) + \int_{G / \Gc} \bigg( \int_{\gamma} \psi \circ F^{2n}_{n+1} \, d (\nu_G)_\gamma \bigg) \, d \nu_G^T \, , 
\end{align*}
where the transversal measure $\nu_G^T$ is the pushforward of $\nu_G$ onto $G / \Gc$.

Fixing $\gamma \in \Gc$, we have by the density estimate in Proposition \ref{prop:collectionFullyCrossingCurves} that
\[
\int_{\gamma} \psi \circ F^{2n}_{n+1} \, d (\nu_G)_\gamma = (1 + O(L_n^{1 - 2 \eta})) \int_\gamma \psi \circ F^{2n}_{n+1} \, d \Leb_\gamma \, ,
\]
and so applying Proposition \ref{prop:CLTdocCurves} with $m \mapsto n+1, n \mapsto 2n$, we have
\begin{gather*}
\int_{\gamma} \psi \circ F^{2n}_{n+1} \, d (\nu_G)_\gamma\\
=
(1 + O(L_n^{1 - 2 \eta})) \Len(\gamma) \cdot \int \psi + (1 + O(L_n^{1 - 2 \eta})) \| \psi \|_\a \cdot O \bigg( L_{2n}^{-\a(1-\eta)/(2 + \a)} + L_{n+1}^{1 - 2 \eta} + \sum_{k = n+1}^{2n-1} L_k^{-1 + \eta} \bigg) \\
= \int \psi + \| \psi \|_\a \cdot O \bigg( L_{2n}^{-\a(1-\eta)/(2 + \a)} + L_{n}^{1 - 2 \eta} + \sum_{k = n+1}^{2n-1} L_k^{-1 + \eta} \bigg) \, .
\end{gather*}

Collecting these estimates, we conclude
\[
\bigg| \int \psi \circ F^{2n} \, d \nu - \int \psi \bigg| \leq C \| \psi \|_\a  \bigg( L_n^{- \min\{ 2 \eta -1, \a (1 - \eta) /(2 + \a) \}} +   \ell^{-2} \sum_{i = \lfloor n/4 \rfloor}^\infty L_i^{- \frac12(1 - \eta)} \bigg) \, .
\]
This completes the proof.

%\newpage

\section*{Appendix}

\begin{lem}[Partition saturation]\label{lem:partitionSaturation}
Let $X$ be a compact metric space, $\Bor(X)$ the Borel $\sigma$-algebra on $X$, and let $\mu$ be a probability on $(X, \Bor(X))$. Let $\xi$ be a measurable partition of $X$, and denote by $(\mu_C)_{C \in \xi}$ the canonical disintegration of $\mu$ with respect to $\xi$. Let $\mu^T$ denote the transverse measure on $X / \eta$.

Let $Y \in \Bor(X)$. Then, $\mu^T \{ C \in X/\eta : \mu_C(Y) > 0 \} \geq \mu(Y)$.
\end{lem}
\begin{proof}
We estimate
\begin{align*}
\mu(Y) & = \int_{X / \eta} \mu_C(Y) \, d \mu^T(C) 
= \int_{ C \in X / \eta : \mu_C(Y) > 0} \mu_C(Y) \, d \mu^T(C) \\
& \leq \mu^T \{ C \in X / \eta : \mu_C(Y) > 0 \} \, . \quad  \qedhere
\end{align*}
\end{proof}

\bibliographystyle{plain}
\bibliography{biblio_nonauto}

\end{document}